\documentclass[12pt, a4paper]{amsart}

\usepackage{eurosym}
\usepackage[dvipsnames]{xcolor}
\usepackage{amsmath,amssymb,amsthm,amsfonts}
\usepackage[mathscr]{eucal}
\usepackage{fullpage}
\usepackage{graphicx}
\usepackage{subfig}
\usepackage{epstopdf}
\usepackage{lscape}
\usepackage{setspace}
\usepackage{epsfig}
\usepackage[round]{natbib}
\usepackage{multirow}
\usepackage{color}
\usepackage[T1]{fontenc}
\usepackage[utf8]{inputenc}
\usepackage{bbm}
\usepackage{mathtools}
\usepackage{appendix}
\usepackage{comment}
\usepackage{enumitem}
\usepackage[normalem]{ulem}
\allowdisplaybreaks

\newtheorem{theorem}{Theorem}[section]
\newtheorem{corollary}{Corollary}[section]
\newtheorem{lemma}{Lemma}[section]
\newtheorem{proposition}{Proposition}[section]
\newtheorem{definition}{Definition}[section]
\newtheorem{remark}{Remark}[section]
\newtheorem{example}{Example}

\newcommand{\ba}{\begin{eqnarray}}
\newcommand{\ea}{\end{eqnarray}}
\def\XX{\boldsymbol{X}}
\def\xx{\boldsymbol{x}}

\def\KK{\boldsymbol{K}}

\def\AA{\boldsymbol{A}}

\def \UU{\boldsymbol{U}}

\def\VV{\boldsymbol{V}}
\def\SS{\boldsymbol{S}}
\def\ss{\boldsymbol{s}}
\def\WW{\boldsymbol{W}}
\def\w{\boldsymbol{w}}
\def\uu{\boldsymbol{u}}

\def\tt{\boldsymbol{t}}
\def\aa{\boldsymbol{a}}

\def\YY{\boldsymbol{Y}}

\def\yy{\boldsymbol{y}}
\def\xxi{\boldsymbol{\xi}}

\def\BB{\boldsymbol{B}}
\def\Si{\boldsymbol{\Sigma}}

\def\GGG{\mathcal{G}}
\def\ets{\mathcal{E}T\alpha S}

\def\qq{\boldsymbol{q}}
\def\ww{\boldsymbol{w}}

\def\mmu{\boldsymbol{\mu}}

\def\RR{\mathbb{R}}
\def\rr{\boldsymbol{r}}

\def\ZZ{\boldsymbol{Z}}

\def\La{\boldsymbol{\Lambda}}
\def\uu{\boldsymbol{u}}
\def\vv{\boldsymbol{v}}

\def\ee{\boldsymbol{e}}

\def\SSigma{\boldsymbol{\Sigma}}

\def\diag{\text{diag}}

\def\FFF{\mathcal{F}}

\def\banach{{C}_0(\RR^d)}
\def\banachb{{B}_b(\RR^d)}

\newcommand{\normsup}[1]{\| #1 \|_{\infty}}

\begin{document}

 \title{ 
Additive subordination of multiparameter Markov processes
}

    \author[G. D'Onofrio]{Giuseppe D'Onofrio$^{\circ}$}
        \address{$^{\circ}$ Dipartimento di Scienze Matematiche, Politecnico di Torino, Corso Duca degli Abruzzi 24, 10129 Torino, Italy}
        \email{giuseppe.donofrio@polito.it}

    \author[A.Mutti]{Alessandro Mutti$^{\circ}$}
              \email{alessandro.mutti@polito.it}

    \author[\v P.Semeraro]{ Patrizia Semeraro$^{\circ}$}
\email{patrizia.semeraro@polito.it}

\begin{abstract}
 We extend Phillips theorem to  multiparameter, multidimensional Markov processes that are time-changed by an independent additive subordinator. 
We show that the resulting process is a Feller evolution and we characterize its generator.   
    We further derive its pseudo-differential representation and show that its symbol admits a L\'evy-Khintchine representation, thereby providing analytical tractability.
We then subordinate a multiparameter Ornstein–Uhlenbeck process using a multivariate Sato process and construct a multivariate asset model that accounts for both mean reversion and time inhomogeneity, allowing each asset to evolve according to its own economic time, in line with empirical evidence.

\end{abstract}

\medskip
    
\noindent\keywords{Multiparameter Ornstein-Uhlenbeck processes, Sato processes, multivariate inhomogeneous subordination, multivariate asset modeling, Lévy-type processes}
\subjclass{60G53, 60J35, 47D07}

\maketitle

\section{Introduction} 
\label{sec:level1}

Subordinated L\'evy models are widely used in finance to model asset returns, motivated by both theoretical and empirical considerations.
From the theoretical perspective, they are semimartingales,  are parsimonious in terms of parameters, and  can be specified to admit analytical characteristic functions.
They are also economically meaningful, as the subordinator can be interpreted as a stochastic economic clock, measured through the volume of trades (\cite{ane2000order}).
From the empirical point of view, they are able to capture key stylized features of asset returns, such as skewness and  kurtosis (\cite{cont2001empirical}).

In the multivariate setting, assuming a single economic time for all assets is unrealistic, as cross-sectional trading characteristics differ significantly across assets (\cite{Ha}). 
This motivates the need to assign each asset its own economic clock, leading to the construction of Lévy processes via multiparameter subordination à la  \cite{ba};  see also \cite{Sem1}, \cite{LuciSem1}, \cite{guillaume2018multivariate}, \cite{buchmann2019weak} and \cite{meoli2025bivariate}, just to name a few.
However, when using L\'evy models, the log-return processes exhibit independent and stationary increments, which limits their ability to reproduce the variations of the implied volatility
smile and the implied correlation across maturities. A parsimonious generalization that overcomes these limitations is given by additive subordination. 
Additive subordinators introduce time inhomogeneity while preserving analytical tractability.
Additive subordination was introduced by \cite{li2016additive} for general Markov processes. 
Their results allow additive subordination  of It\^o processes, which are widely used in financial modeling. 
For instance, \cite{li2013ornstein} and \cite{li2014time} consider the Ornstein-Uhlenbeck process,  which is well suited to model  the mean-reverting behavior typical of commodity markets.
However, an important limitation of their construction is that the additive subordinator is one-dimensional. 
To the best of our knowledge, multiparameter additive subordination, allowing each asset to evolve according to its own clock, has so far been explored only in the context of Lévy processes.

This paper generalizes the results in \cite{li2016additive} and \cite{mendoza2016multivariate} to additive subordination of multiparameter Markov processes, the latter defined in \cite{khoshnevisan2006multiparameter}. 
We generalize Phillips theorem and derive the symbol of the subordinated process, which serves as the analogue of the characteristic exponent in the Lévy process framework and, under suitable regularity conditions, admits a Lévy-Khintchine representation.
We then focus on multiparameter additive subordination of the Ornstein-Uhlenbeck process, providing a multidimensional extension of the results in \cite{li2013ornstein}.

We further specify the model with an application to the energy market in mind. 
To maintain parsimony in the number of parameters, we consider Sato subordinators. 
Sato processes are additive processes used in financial modeling due to their ability to capture the term structure of moments observed in financial markets (see, e.g., \cite{carr2007self}, \cite{guillaume2018multivariate}, \cite{carr2021additive}, and \cite{azzone2025explicit}).
The case of Sato subordination applied to multiparameter Brownian motions was introduced in \cite{semeraro2020note}.
The ability of Sato subordination to fit the term structures of both correlation and volatility in the stock market is demonstrated by \cite{amici2025multivariate} and supports our choice.

Beyond its financial applications, multiparameter additive subordination provides a simple and powerful method for constructing time-inhomogeneous Markov semimartingales with jumps.

The definition of multiparameter processes is not unique. In addition to the approach of \cite{ba}, many authors employ the concept of L\'evy sheets, see for instance \cite{barndorff2012meta}, \cite{pedersen2004relations}; and refer to \cite{khoshnevisan2006multiparameter} for a comparison of the two approaches.
Although the latter has an elegant mathematical formulation, it is not well suited for our purposes.
Multiparameter subordination was introduced by \cite{ba} for L\'evy processes. 
In this framework, multiparameter L\'evy processes $L(\ss)$ are parametrized by a vector parameter $\ss\in \RR^k_+$, and defined such that the one-parameter processes $L(s\ee_j)$ can represent different time scales, where $\{\ee_1,\dots,\ee_k\}$ is the canonical basis of $\RR^k$. 
This allows for the construction of a system in which each marginal process evolves according to its own time scale with the parameter $\ss$ functioning as a vector-valued time. 
The price to pay for this property is that the system of distributions $(\XX(\ss_1),\ldots, \XX(\ss_n))$  along the directions given by the componentwise partial order on $\RR^k$  does not identify the whole system of finite dimensional distributions (see Remark 4.6 in \cite{ba}).
However, it is sufficient to identify the finite dimensional distributions of the one-parameter process obtained by subordination. 
On the other hand, L\'evy sheets are also processes parametrized by a vector parameter $\ss\in \RR^k_+$, but they are not well suited to represent one-dimensional processes with different time scales, since if $L(\ss), \ss\in \RR^k_+$ is a L\'evy sheet then, for all $j=1,\ldots, k$ we have $L(s\ee_j)=0$. Furthermore,  subordination of L\'evy sheets on $\RR^k_+$ (\cite{barndorff2012meta}) is defined with the aim of an inner operation: the subordinated process remains a Lévy sheet and thus retains its vector-parameter structure.
These considerations motivate the use of additive subordination of multiparameter Markov processes in the spirit of \cite{ba} in mathematical finance.

The paper is organized as follows. 
Section \ref{Markov} recalls Markov processes and Feller evolutions together with generators and symbols. Moreover, the concept of multiparameter Markov processes and their properties is introduced.
Then, in Section \ref{sec:MultiparameterAssitiveSUB}, we introduce multiparameter additive subordination, we prove the generalization of Phillips Theorem, and study the symbol of the resulting Feller evolution. 
The multiparameter Ornstein-Uhlenbeck additive subordination is considered in Section \ref{sec:M-OUprocess}, where the Lévy-Khintchine representation of the symbol of this process is explicitly derived. A first {mean reverting, time inhomogeneous multivariate asset model is provided in Section \ref{sec:fin}. }
This example relies on the Sato subordination. 
Finally, we discuss possible {future real world applications} of our study in Section \ref{sec:FutureWorks}.

\section{Markov processes, characteristic function, and symbol}\label{Markov}

Let $(\Omega,\mathcal{A},\mathbb{P})$ be a probability space and let $(\FFF_t,t \geq 0)$ be a filtration.
An $\FFF_t$-adapted $\RR^d$-valued stochastic process $(\XX(t))_{t \geq 0}$ is a Markov process if 
\begin{equation*}
    \mathbb{E}(f(\XX(t)) \vert \FFF_s) = \mathbb{E}(f(\XX(t)) \vert \XX(s)),
\end{equation*}
for every $0 \leq s \leq t < \infty$ and $f \in \banachb$, the space of all bounded Borel measurable functions from $\RR^d$ to $\RR$.

With each Markov process $\XX(t)$ we associate a family of two parameter operators $(T_{t_1,t_2})_{ t_1\leq t_2}$, also called propagators, on the Banach space  equipped with the supremum norm by
\begin{equation}\label{eq:oper}
    T_{t_1,t_2}f(\xx) = \mathbb{E}[f(X(t_2)) | X(t_1) = \xx],
\end{equation}
for every $f \in \banachb$, $\xx \in \RR^d$.

The Markov process is homogeneous if we have $T_{s,s+t} = T_{0,t}$ for every $s,t \geq 0$. In this case, we write $T_t = T_{0,t}$, for every $t \geq 0$, and the family of two parameter operators reduces to a semigroup $(T_t)_{t\geq0}$.
We assume that the Markov process is normal, i.e. $T_{t_1,t_2}(\banachb)\subseteq\banachb.$ This implies, by Theorem 3.1.2 in~\cite{applebaum2009levy}, that the family of operators in \eqref{eq:oper} is a Markov evolution system, defined by the following.

\begin{definition} \label{def:MarkovEvol}
A family of linear operators  $(T_{t_1,t_2})_{ t_1\leq t_2}$ is a Markov evolution if it satisfies the following conditions:
\begin{enumerate}
    \item $T_{t_1,t_2}$ is a linear operator on $\banachb$ for every $0 \leq t_1 \leq t_2 < \infty$;
    \item $T_{t,t} = \textrm{Id}$ for every $t \geq 0$;
    \item $T_{t_1,t_2}T_{t_2,t_3} = T_{t_1,t_3}$ for each $0 \leq t_1 \leq t_2 \leq t_3 < \infty$;
    \item $f \geq 0 \implies T_{t_1,t_2}f \geq 0$ for every $0 \leq t_1 \leq t_2 < \infty$, $ f \in \banachb$;
    \item $T_{t_1,t_2}$ is a contraction, i.e. $\normsup{T_{t_1,t_2} f} \leq \normsup{f}$ for every $0 \leq t_1 \leq t_2 < \infty$, $ f \in \banachb$;
    \item $T_{t_1,t_2}1 = 1$ for all $0 \leq t_1 \leq t_2 < \infty$.
\end{enumerate}
\end{definition}
A Markov evolution system that is strongly continuous, i.e.,
\begin{equation}\label{eq:feller}
    \lim_{(s,t)\rightarrow (v, w), \, s\leq t}
    \|T_{s,t} f-T_{v,w} f\|_{\infty} = 0,
\end{equation}
and satisfies $T_{t_1,t_2}(\banach)\subseteq \banach$, where $\banach$ denotes the set of continuous functions from $\RR^d$ to $\RR$, vanishing at infinity, is called a Feller evolution system.

Corresponding to an evolution system the family of right generators is defined by (\cite{bottcher2014feller})
\begin{equation}
    \mathcal{G}_{t}^+ f = \lim_{h\rightarrow 0+} h^{-1}(T_{t, t+h}f - f),
\end{equation}  
similarly, the family of left generators is defined as 
\begin{equation}
    \mathcal{G}_{t}^- f = \lim_{h\rightarrow 0+}h^{-1}(T_{t-h, t} f - f).
\end{equation}  
Following \cite{li2016additive}, we consider the right generators that we simply call generators and denote by $\mathcal{G}_t:=\mathcal{G}^+_t$. If the Markov process is homogeneous the Feller evolution reduces to a Feller  semigroup with generator 
\begin{equation}
    \mathcal{G} f = \lim_{t\rightarrow 0}t^{-1}(T_{t} f - f).
\end{equation}

\begin{remark}
\label{rem_spaces}
    The preceding definitions and many of the results that follow can be stated in greater generality on Banach spaces of functions on a locally compact, separable metric space \(E\).
Given the applications we have in mind, and for the sake of readability, we shall take
$
E=\mathbb{R}^d
$
and work with the Banach space $C_0(\mathbb{R}^d)$, the set of continuous functions from $\RR^d$ to $\RR$, that vanish at infinity,  equipped with the supremum norm $
\|f\|_\infty = \sup_{x\in\mathbb{R}^d}|f(x)|.
$
Finally, when dealing with pseudo-differential operators, we shall consider the space
$\mathcal{C}_c^{\infty}(\mathbb{R}^d),
$
that is, the space of infinitely differentiable functions on $\mathbb{R}^d$ with compact support. 
\end{remark}

Let $\mu_{t}(\xx, d\yy)=P(\XX(t)\in d\yy|\XX(0)=\xx)$.  With a slight abuse of notation, we denote by $\hat{\mu}_{t}(\xx, \boldsymbol{\xi})$ the characteristic function of the random variable $\XX(t)-\xx$, that is
\begin{equation}
    \hat{\mu}_{t}(\xx, \boldsymbol{\xi}) = \mathbb{E}_{\xx}[e^{i(\XX(t)-\xx) \cdot \boldsymbol{\xi}}],
\end{equation}
where $\mathbb{E}_{\xx}[\cdot] = \mathbb{E}[\cdot|\XX(0)=\xx].$
If $\XX$ is a L\'evy process we write $\hat{\mu}_{t}( \boldsymbol{\xi})=\hat{\mu}_{t}(\boldsymbol{0}, \boldsymbol{\xi})$, which justifies our notation. 

The following definition generalizes the characteristic exponent of a Lévy process to the Feller setting.

\begin{definition}[Definition 1.6 \cite{jacob2001levy}]
    For a Feller process $\XX(t)$ the function
    \begin{equation}
        q(\xx, \xxi) = 
        -\lim_{t\rightarrow 0^+}\frac{\mathbb{E}_{\xx} [e^{i(\XX(t)-\xx) \cdot \xxi}]-1}{t}
        =
        \left. - \frac{d}{dt}\hat{\mu}_{t}(\xx, \boldsymbol{\xi}) \right\vert_{t=0}
    \end{equation}
    is called the symbol of the process, {whenever the limit exists}.
   \end{definition}
Although the symbol of a Feller process does not coincide in general with the characteristic exponent, it can be used to describe many probabilistic properties of the process; for example, sample path properties.
In this case, the symbol of a Feller process assumes the same role as the characteristic exponent of a Lévy process (\cite{schilling2016introduction}).
Nevertheless, the short-time path behavior of a Feller process resembles that of a Lévy process with characteristic exponent coinciding with the symbol $q(\xx, \xxi)$. 
This local behavior is the reason why such processes are often referred to as \emph{Lévy-type processes}.

From Theorem 1.4 in \cite{jacob2001levy}, if $\XX(t)$ is a Feller process, the operator semigroup $(T_{t})_{t \geq 0}$ restricted to the space $\mathcal{C}^{\infty}_c(\RR^d)$ is a family of pseudo-differential operators, that is, they have the representation
\begin{equation}\label{eq:ChTi}
    T_{t} f(\xx) = (2\pi)^{-d/2} \int_{\RR^d}e^{i \xx \cdot \boldsymbol{\xi}} \hat{\mu}_{t}(\xx, \boldsymbol{\xi})\hat{f}(\boldsymbol\xi)d\boldsymbol\xi,
\end{equation}
where $\hat{f}$ denotes the Fourier transform of $f$, $\hat{f}(\xxi) = (2\pi)^{-\tfrac{d}{2}} \int_{\RR^d} e^{-i\xx \cdot \xxi} f(\xx) d\xx$. 

Thanks to the Feller property of $\XX(t)$, the symbol of the process (whenever it exists) has the following L\'evy-Khintchine representation (\cite{jacob2001levy})
\begin{equation}
\begin{split}
    q(\xx, \boldsymbol{\xi})
    &=
    -i \boldsymbol{\gamma(\xx)} \cdot \boldsymbol{\xi} +\frac{1}{2}\boldsymbol{\xi} \cdot {\Sigma(\xx)}\boldsymbol{\xi} -\int_{\RR^d}(e^{i\xxi \cdot \yy}-1-i \xxi \cdot \yy \mathbbm{1}_{B}(\yy)){\nu}(\xx,d\yy),
    \end{split}
    \end{equation}
where $\mathbbm{1}$ is the indicator function, $B$ is the unit ball in $\RR^d$, and where for any $\xx \in \RR^d$, $\gamma(\xx)$ is a deterministic drift, $\Sigma(\xx)$ is a symmetric non-negative definite matrix, and ${\nu}(\xx,d\yy)$ is a L\'evy measure.
The generator $\mathcal{G}$ of the semigroup $(T_t)_{t \geq 0}$ associated to $\XX(t)$ has the following pseudo-differential representation
\begin{equation} \label{eq:pseudo-diff-repr}
    \mathcal{G} f(\xx)
    =
    -(2\pi)^{-d/2}
    \int_{\RR^d} e^{i\xx \cdot \boldsymbol{\xi}} q(\xx, \boldsymbol{\xi}) \hat{f}(\boldsymbol\xi)d\boldsymbol\xi.
\end{equation}

\subsection{Multiparameter Markov process}

Let us introduce the partial ordering $\preceq$ on $\RR^k$, such that for $\ss, \tt \in \RR^k$, we have $\ss \preceq \tt$ if $s_j \leq t_j$, for every $j \in \{1,\dots,k\}$.
Given a probability space $(\Omega,\mathcal{A},\mathbb{P})$, the family $\mathcal{F} = (\mathcal{F}_{\ss}, \ss \in \RR^k_+)$ is a $k$-parameter filtration if, for every $\ss \in \RR^k_+$, $\mathcal{F}_{\ss}$ is a $\sigma$-algebra and $\mathcal{F}_{\ss} \subseteq \mathcal{F}_{\tt} \subseteq \mathcal{A}$, for every $\ss,\tt \in \RR^k_+$ such that $\ss \preceq \tt$.

We refer to \cite{jacob2010multiparameter} for multiparameter  semigroups of probability measures,  multiparameter semigroups of operators, and their links. 
\begin{definition}
    A $k$-parameter family $(T_{\ss})_{\ss\succeq \boldsymbol{0}},$ of bounded linear operators $T_{\ss}$ on $\banach$ is called a $k$-parameter semigroup of operators if $T_{\boldsymbol{0}}=Id$ and
     $T_{\ss+\tt}=T_{\ss}T_{\tt}$.
\end{definition}
We assume throughout the paper that the semigroup $(T_{\ss})_{\ss\succeq \boldsymbol{0}}$ is strongly continuous on $\banach$ and conservative.

\begin{definition} [\cite{khoshnevisan2006multiparameter}]\label{def:HomogeneousMArkovProcess}
    A $k$-parameter stochastic process $(\XX(\ss),\ss \in \RR^k_+)$ is a \emph{homogeneous Markov process} if there exist a $k$-parameter filtration $\FFF = (\FFF_{\ss}: \ss \in \RR^k_+)$ and a family of operators $(T_{\ss})_{\ss\succeq0}$ such that, for all $\xx \in \RR^d$, there exists a probability measure $\mathbb{P}_{\xx}$ such that
    \begin{enumerate}[label=\textup{(\alph*)}]
        \item $(\XX({\ss}))_{\ss \in \RR^k_+}$ is adapted to $(\FFF_{\ss})_{\ss \in \RR^k_+}$;
        \item The trajectories $\ss\mapsto \XX(\ss)$ are $\mathbb{P}_{\xx}$ a.s. c\`adl\`ag --- in the partial ordering $\preceq$\@;
        \item For all $\ss\in \RR^k_+$, $\FFF_{\ss}$ is $\mathbb{P}_{\xx}$ complete. Moreover $\mathcal{F}$ is a commuting filtration with respect to all measures $\mathbb{P}_{\xx}.$
        \item \label{point:LinOperator} $\mathbb{E}_{\xx}[f(\XX(\ss+\tt)) \vert \mathcal{F}_{\ss}] = T_{\tt} f(\XX(\ss))$, for every $f \in C_0(\mathbb{R}^d)$, $\ss,\tt \in \RR^k_+$; 
        \item $\mathbb{P}_{\xx}(\XX(\boldsymbol{0})=\xx) = 1$. 
    \end{enumerate}
\end{definition}
\begin{remark}
    We observe that a $k$-parameter stochastic process $(\XX(\ss),\ss \in \RR^k_+)$ is a random field as a natural generalization of a stochastic process in which the time interval is replaced by a different set (see \cite{applebaum2009levy}, section 1.1.8 or \cite{khoshnevisan2006multiparameter}).
\end{remark}

The multiparameter L\'evy process introduced in \cite{ba} is an example of a stochastically continuous multiparameter Markov process, see Example \ref{ex:1} for the semigroup associated to a multiparameter L\'evy process. Another important example is the following.
\begin{example}\label{ex:sum}
    Let $\XX_j(s_j)$, for $j \in \{1,\dots,k\}$, be independent Markov processes, then, the process $\XX(\ss;k) = (\XX_1(s_1), \dots, \XX_k(s_k))$ is a $k$-parameter Markov process. 
    The proof is provided in Section 2.2 of Chapter 11 in \cite{khoshnevisan2006multiparameter}.
\end{example}
From Item \ref{point:LinOperator} in Definition \ref{def:HomogeneousMArkovProcess}, it follows that, for $\ss \in \RR^k_+$ and $\xx \in \RR^d$,
\begin{equation*}
    T_{\ss}f(\xx) = 
    \mathbb{E}_{\xx}[T_{\ss}f(\XX(\boldsymbol{0}))]
    =
    \mathbb{E}_{\xx}[ \mathbb{E}_{\xx} (f(\XX(\boldsymbol{s}))\vert \FFF_{\boldsymbol{0}})]
    =
    \mathbb{E}_{\xx}[f(\XX(\boldsymbol{s}))]
    =
    \int_{\RR^d}f(\yy)\mu_{\ss}(\xx, d\yy),
\end{equation*}
where $\mu_{\ss}(\xx,d\yy) = \mathbb{P}_{\xx} (\XX(\ss)\in d\yy )$.
It holds $T_{\ss+\tt} = T_{\ss}  T_{\tt}$ for any $\ss,\tt \in \RR^k_+$, and the family of operators $(T_{\ss})_{\ss\succeq0}$ is a $k$-parameter semigroup of operators on $C_0(\mathbb{R}^d)$.
The kernels satisfy
\[
\mu_{\ss+\tt}(\xx,B)=\int_{\RR^d}\mu_{\tt}(\yy,B)\mu_{\ss}(\xx,d\yy).
\]

\begin{remark}As mentioned in the introduction, two Markov processes with the same semigroup of operators do not necessarily have the same finite dimensional joint distributions, since the system of distributions $(\XX(\ss_1),\ldots, \XX(\ss_n))$ with $\ss_1 \preceq \ss_2 \preceq \ldots \preceq \ss_n$, $n \in \{1,2,\dots\}$, does not determine the whole system of finite dimensional distributions.
However, the system of distributions $(\XX(\ss_1),\ldots, \XX(\ss_n))$  along the directions given by the partial order is sufficient to identify the finite dimensional distributions of a process obtained by subordination.
\end{remark}

An important result for the practical construction of time inhomogeneous subordinated Markov processes is that multiparameter semigroups factorize into direct product of commuting one parameter semigroups (see \cite{mendoza2016multivariate}, \cite{butzer2013semi} and \cite{baeumer2008subordinated}).

\begin{proposition}[{\cite{mendoza2016multivariate}}]\label{prop:prod}
If $(T_{\ss})_{\ss\succeq 0}$ is a $k$-parameter strongly continuous semigroup on  $\banach$ then it is the direct product of $k$ one-parameter strongly continuous semigroups $(T_{s}^{(j)})_{s\geq 0}$, $j=1,\ldots, k$, on $\banach$ with infinitesimal generator $\mathcal{G}^{(j)}$ with domain $\text{Dom}(\mathcal{G}^{(j)})\subset{\banach} $ and the semigroups commute with each other. The set of generators $(\mathcal{G}^{(j)}, \text{Dom}(\mathcal{G}^{(j)}))$ is called set of generators of the $k$-parameter semigroup $(T_{\ss})_{\ss\geq 0}$.
\end{proposition}
Let us define  $\prod_{j=1}^kT^{(j)}_{s_j}:=T^{(1)}_{s_1}\cdots T^{(k)}_{s_k}$ where $T^{(1)}_{s_1}\cdots T^{(k)}_{s_k}$ stands for composition of commuting operators. 
It follows from Proposition \ref{prop:prod} that
\begin{equation} \label{eq:MarginalSemigroups}
    T_{\ss}=\prod_{j=1}^kT^{(j)}_{s_j},\quad \ss\in\RR^k_+,
\end{equation}
and we also notice that 
\begin{equation} \label{eq:splitT2}
    T^{(j)}_{s_j}f(\xx)=T_{s_j\cdot \ee_j}f(\xx),
\end{equation} 
where $\{\ee_1, \dots, \ee_k\}$ is the canonical basis of $\RR^k$, see \cite{khoshnevisan2006multiparameter}. The one-parameter operators $T_{s_j}^{(j)}$ are called marginal transition operators.
\begin{example}{\bf Lévy semigroups.}
\label{ex:1}
For an overview of operator semigroups and L\'evy processes see \cite{applebaum2009levy}.
We refer to  \cite{pedersen2004relations} for L\'evy multiparameter processes and multiparameter convolution semigroups of probability measures.  
Any $k$-parameter  convolution semigroup $(\mu_{\tt})_{\tt\in \RR^k_+}$  gives rise to a $k$-parameter L\'evy process and therefore to a $k$-parameter operator semigroup $(T_{\ss})_{\ss\succeq 0}$
\begin{equation}\label{eq:MultiSemi}
    T_{\ss} f(\xx) = E[f(\xx+\XX_{\ss})] = \int_{\RR^d} f(\xx+\yy)\mu_{\ss}(d\yy),
\end{equation}
for all $\ss\in \RR^k_+$ and $f \in \banach$. 
Indeed, $(T_{\ss})_{\ss\succeq 0}$ is a strongly continuous contraction semigroup on {$\banach$} which is positivity preserving.
Hence, it is a Feller semigroup (see \cite{jacob2010multiparameter}).
    Let $\XX^{(j)}(s_j)$, for $j\in \{1,\ldots,k\}$, be independent L\'evy processes, then $\XX(\ss) = \sum_{j=1}^k \XX^{(j)}(s_j)$ is a $k$-parameter Markov process, as proved in \cite{ba}.
Let ${{\BB}}^{(j)}(t)$ be independent Brownian motions on
$\mathbb{R}^{d}$ with drift ${\boldsymbol{\mu}_j}$ and covariance matrix
${\Sigma_{j}}$.  We can define the multiparameter L\'evy
process 
\begin{equation}
    \BB(\ss) = \sum_{j=1}^k{\BB}^{(j)}(s_j), \quad {s}\in\mathbb{R}^{k}_{+},
\end{equation}
that we call multiparameter Brownian motion.
We have
\begin{equation}
    \BB(s_j \ee_j) = \BB^{(j)}(s_j).
\end{equation}
\end{example}
\section{Multiparameter Additive subordination} \label{sec:MultiparameterAssitiveSUB}
This section introduces time-inhomogeneous additive subordinators. An additive subordinator is an increasing process with non stationary independent increments and continuous in probability. Additive processes are obtained from L\'evy ones by relaxing the condition of stationary increments, hence they are spatially (but not temporally)
homogeneous (see \cite{sa}).

Let $\SS(t)$ be an additive subordinator
and let $\pi_{t_1,t_2}$ be the distribution of its increment $\SS(t_2)-\SS(t_1)$. An additive process is completely determined by the set of measures $\pi_{t_1,t_2}$ since all the finite-dimensional distributions are completely specified (\cite{beghin2019additive}). 
Adapting Proposition 2.1 in \cite{li2016additive} we have that an additive subordinator is a semimartingale with a time dependent characteristic exponent. 
In fact, the time $t$ characteristic function of an additive subordinator $\SS(t)$ is  (\cite{mendoza2016multivariate})
\begin{equation}\label{CHSbv}
    \hat{\mu}^S_t(\xxi)
    =
    \exp \left\{ i \boldsymbol{b}(t) \cdot \xxi +\int_{\RR_+^k \setminus \{\boldsymbol{0} \}} (e^{i \xx \cdot \xxi}-1)\eta(t, d\xx) \right\}, \quad \xxi \in \RR^k,
\end{equation}
where $\boldsymbol{b}(t)\in \RR^k_+$ is the time dependent drift and  $\eta(t,d\xx)$ is  a  time-dependent L\'evy measure. We assume that $\boldsymbol{b}(t)$ and $\eta(t, d\xx)$ are absolutely continuous with respect to $t$ with differential characteristics $\boldsymbol{c}(t)$ and $\nu(t, d\xx)$, respectively.
Although they are not Lévy processes, they somehow generalize subordinators in the sense that their Laplace exponents are possibly different
Bernstein functions for each time $t$. 
Because of time-inhomogeneity, additive subordinators sometimes go under the name of non-homogeneous subordinator (see \cite{orsingher2016time}) and define two-parameter semigroups (or propagators).

We want to focus on the additive subordination of the multiparameter Markov evolution defined in the previous section. To this aim, a key tool for studying the resulting time-changed process is the derivation of an appropriate Phillips Theorem.

The multiparameter version of Phillips Theorem has been proved by \cite{baeumer2008subordinated}. 
The next result is an extension to multiparameter additive subordination. 
In particular, it is a multiparameter version of Theorem 3.1 in \cite{li2016additive} and a generalization of Theorem 2.3 in \cite{mendoza2016multivariate} to additive subordinators.

\begin{theorem}\label{teo:sub}
    Let $(\SS(t))_{t\geq0}$ be a $k$-dimensional additive subordinator.
    Let $(T_{\ss})_{\ss \succeq \boldsymbol{0}}$ be a $k$-parameter strongly continuous semigroup on $\banach$, associated to a multiparameter Markov process $\XX(\ss)$ independent of $\SS(t)$. The two-parameter family of operators $(T_{t_1, t_2})_{t_1\leq t_2}$,
    \begin{equation} \label{eq:DefOpY}
        T_{t_1, t_2} f(\xx) = \int_{\RR^k_+} T_{\ss}f(\xx) \, \pi_{t_1,t_2}(d\ss),
    \end{equation}
    for $f \in \banach$, where the probability measure $\pi_{t_1,t_2}(d\ss)$ is the distribution of the increment $\SS({t_2})-\SS({t_1})$, is a Markov evolution associated to the subordinated process
    \begin{equation}\label{eq:asp1}
        \YY(t)=\XX(\SS(t)), \quad t\geq 0,
    \end{equation}
    that is a time inhomogeneous normal Markov process with transition {kernels} given by $p_{t_1,t_2}(\xx,d\yy) = \int_{\RR^k_+} \mu_{\ss}(\xx,d\yy) \pi_{t_1,t_2}(d\ss)$, where $\mu_{\ss}(\xx,d\yy)$ is the law of $\XX(\ss)$ starting from $\xx \in \RR^d$.
\end{theorem}
\begin{proof}
We first verify that $(T_{t_1,t_2})_{t_1 < t_2}$ is a Markov evolution according to Definition~\ref{def:MarkovEvol}.
Linearity, positivity, contractivity, and $T_{t_1,t_2}1 = 1$ directly follow from the corresponding properties of the semigroup $(T_{\ss})_{\ss \succeq \boldsymbol{0}}$.
Moreover, since $\pi_{t,t} = \delta_{\boldsymbol{0}}$, we have $T_{t,t}f = T_{\boldsymbol{0}}f = f$.
By Fubini's theorem and the semigroup property of $(T_{\ss})_{\ss \succeq \boldsymbol{0}}$, for $t_1 \leq u \leq t_2$, we have
\begin{equation}
\begin{split}
    T_{t_1,u} T_{u,t_2} f(\xx) 
    &= 
    \int_{\RR^k_+} \int_{\RR^k_+} T_{\ss} T_{\rr} f(\xx) \pi_{u,t_2}(d \rr) \pi_{t_1,u}(d \ss)
    \\ &=
    \int_{\RR^k_+} \int_{\RR^k_+} T_{\ss+\rr} f(\xx) \pi_{u,t_2}(d \rr) \pi_{t_1,u}(d \ss)
    \\ &= 
    \int_{\RR^k_+} T_{\vv} f(\xx) \pi_{t_1,t_2}(d \vv) = T_{t_1,t_2} f(\xx),
\end{split}
\end{equation}
where the last equality follows from independence of increments of $\SS$, which implies $\pi_{t_1,t_2} = \pi_{t_1,u}*\pi_{u,t_2}$.
The transition probabilities
\[
p_{t_1,t_2}(\xx,A) = \int_{\mathbb R_+^k} \mu_{\ss}(\xx,A) \, \pi_{t_1,t_2}(d\ss),
\qquad 
A\in \mathcal{B}(\mathbb R^d).
\]
are probability kernels, and Fubini's theorem yields
\begin{equation}\label{eq:trans}
\begin{split}
    T_{t_1, t_2}f(\xx) 
    &= \int_{\RR^k_+}T_{\ss}f(\xx)\pi_{t_1,t_2}(d\ss) 
    = \int_{\RR^k_+}\int_{\RR^d}f(\yy)\mu_{\ss}(\xx, d\yy)\pi_{t_1,t_2}(d\ss)
    \\ &=
    \int_{\RR^d}f(\yy)p_{t_1,t_2}(\xx, d\yy).
\end{split}
\end{equation}
Finally, we show that $p_{t_1,t_2}$ are transition kernels of $\YY(t) = \XX(\SS(t))$.
Consider $0 = t_0 < t_1 < \dots < t_n$ and the Borel sets $A_1,\dots, A_n$. 
Using the independence of $\XX$ and $\SS$, and applying the Markov property of $\XX$, we obtain
\begin{equation}
\begin{split}
    &\mathbb{P}_x
	\bigl(\YY(t_1)\in A_1, \dots, \YY(t_n) \in A_n \bigr)
    \\
	&=
	\int_{\RR_+^k} \cdots \int_{\RR_+^k}
	\mathbb{P}_x
	\bigl(\XX(\ss_1)\in A_1, \dots, \XX(\ss_1 + \dots + \ss_n) \in A_n \bigr)
	\prod_{j=1}^n
	\pi_{t_{j-1},t_j}(d\ss_j)
	\\
	&=
	\int_{\RR_+^k} \cdots \int_{\RR_+^k}
	\left[
	\int_{A_1}\cdots\int_{A_n}
	\mu_{\ss_1}(\xx,d\xx_1)
	\cdots
	\mu_{\ss_n}(\xx_{n-1},d\xx_n)
	\right]
	\prod_{j=1}^n
	\pi_{t_{j-1},t_j}(d\ss_j)
	\\
	&=
	\int_{A_1}\cdots\int_{A_n}
	p_{0,t_1}(\xx,d\xx_1)
	\cdots
	p_{t_{n-1},t_n}(\xx_{n-1},d\xx_n).
\end{split}
\end{equation}
Therefore, $\YY(t)$ is a Markov process with respect to its natural filtration, with transition kernels $p_{t_1,t_2}$.
\end{proof}

The process $\YY(t)$ in \eqref{eq:asp1} is the key object of this work and we refer to it as  \textit{additive subordinated multiparameter Markov process}.
Theorem \ref{teo:sub} and Proposition \ref{prop:prod} lead to
\begin{equation}\label{eq:splitT}
    \begin{split}
        T_{t_1, t_2}f(\xx)=\int_{\RR^k_+}T^{(1)}_{s_1}T^{(2)}_{s_2}\cdots T^{(k)}_{ s_k}f(\xx)\pi_{t_1,t_2}(d\ss).
    \end{split}
\end{equation}

The following theorem provides the generator $\mathcal{G}_t$ of the subordinated process in \eqref{eq:asp1}.
The proof of Theorem \ref{Li} follows the same arguments as the proof of Theorem 3.1 in \cite{li2016additive}, using the multivariate version of Phillips Theorem in \cite{mendoza2016multivariate} instead of Phillips Theorem. See also \cite{mijatovic2010additive}.
Before stating Theorem \ref{Li} we recall the following definition.

\begin{definition}
    A subspace $D$ of the domain $D(\GGG)$ of a linear operator $\GGG: D(\GGG) \subseteq \banach \to \banach$ is called a core for $\GGG$ if $D$ is dense in $D(\GGG)$ for the graph norm $\lVert f \rVert_{\GGG} = \lVert f \rVert_{\infty} + \lVert \GGG f \rVert_{\infty}$.
\end{definition}

\begin{theorem} \label{Li}
Under the assumptions of Theorem~\ref{teo:sub}, consider the restriction of $(T_{t_1,t_2})_{t_1\leq t_2}$ to $C_0(\mathbb R^d)$.
Let $\nu$ be the L\'evy measure of the additive subordinator $\SS(t)$ and define by $\nu_F(t,A) = \int_A( \lVert \ss \rVert \wedge 1) \nu(t, d\ss)$, for any Borel set $A \subseteq \RR^k_+ \setminus \{\boldsymbol{0}\}$, a finite measure on $\RR^k_+ \setminus \{\boldsymbol{0}\}$.
Assume that the left and right limits of $c_j(t)$ exist finite for every $t\geq 0$ and $j=1,\ldots,k$.
Moreover, assume that left and right limits of $\nu_F(t,\cdot)$ converge weakly to a finite measure $\forall t\geq 0$ and that $c_j(t)$ and $\nu_F(t,\cdot)$ are right-continuous and have only a finite number of points of discontinuity. 
The evolution system $(T_{t_1,t_2})_{t_1<t_2}$ in~\eqref{eq:DefOpY} is a strongly continuous backward as well as forward propagator.
For every $t\geq0$, the right generator $\mathcal G_t$ satisfies $\cap_{j=1}^k D(\GGG^{(j)}) \subseteq D(\mathcal G_t)$ and
\begin{equation}\label{eq:generator}
    \mathcal{G}_t f
    =
    \sum_{j=1}^k {c}_j(t) \mathcal{G}^{(j)} f
    +
    \int_{\RR^k_+ \setminus \{\boldsymbol{0}\}}(T_{\ss}f-f) \nu(t, d\ss), 
    \qquad 
    f
    \in \bigcap_{j=1}^kD(\mathcal{G}^{(j)}),
    \end{equation}
    where $\mathcal{G}^{(j)}$ is the generator of the marginal transition operator $(T_s^{(j)})_{s \geq 0}$ in \eqref{eq:splitT2}, for $j \in \{1,\dots,k\}$.
    Moreover, for every $0\leq t_1<t_2$ and $f \in \cap_{j=1}^kD(\mathcal{G}^{(j)})$, we have
    \begin{equation} \label{eq:CommRightGen}
        \lim_{h\rightarrow 0+ } \frac{T_{t_1, t_2+h}f - T_{t_1,t_2}f}{h}
        = 
        T_{t_1,t_2} \mathcal{G}_{{t_2}}f = \mathcal{G}_{{t_2}} T_{t_1,t_2}f.
    \end{equation}  
Finally, $D \coloneqq \cap_{j=1}^k D(\GGG^{(j)})$ is a core for $\mathcal G_t$.
\end{theorem}

\begin{proof}
    Let us call $\SS^{\phi_t}$ the $k$-dimensional Lévy subordinator with drift $\boldsymbol{c}(t)$ and Lévy measure $\nu(t,\cdot)$, and let us denote by $\pi_v^{\phi_t}(d\ss)$ its law at time $v \in \RR_+$. The strongly continuous one-parameter semigroup $(T^{\phi_t}_v)_{v \geq 0}$ corresponding to a homogeneous Markov process $\XX(\ss)$ subordinated by the Lévy subordinator $\SS^{\phi_t}$ is given by (see (2.3) in \cite{mendoza2016multivariate})
    \begin{equation}
    \label{semig_sub_levy}
        T^{\phi_t}_v f = \int_{\RR^k_+} (T_{\ss} f) \pi_v^{\phi_t}(d\ss),
        \quad f \in \banach,
    \end{equation}
    where $(T_{\ss})_{\ss \in \RR_+^k}$ is the operators semigroup associated to the homogeneous Markov process $\XX(\ss)$.
    The generator $\mathcal{G}^{\phi_t}$ of the subordinated process can be obtained using the generalization of the Phillips theorem for multidimensional Lévy-subordinated processes (see for instance \cite{mendoza2016multivariate}, Theorem 2.2 or \cite{baeumer2008subordinated} for a more general setting):
    \begin{equation} \label{eq:Phillips_subLevymultipar}
        \mathcal{G}^{\phi_t} f = \sum_{j=1}^k c_j(t) \mathcal{G}^{(j)} f 
        + 
        \int_{\RR_+^k \setminus \{\boldsymbol{0}\}}(T_{\ss} f - f) \nu(t, d\ss), \quad f \in \bigcap_{j=1}^k D(\mathcal{G}^{(j)}),
    \end{equation}
    where $(\mathcal{G}^{(j)},D(\mathcal{G}^{(j)}))$, for $j \in \{ 1,\dots,k\}$, is the generator corresponding to the marginal semigroup $(T^{(j)}_{s})_{s \geq 0}$ defined in~\eqref{eq:splitT2}.
    For $f \in  \bigcap_{j=1}^k D(\mathcal{G}^{(j)})$, by the assumptions on $c_j$, the map $t\mapsto c_j(t)\mathcal G^{(j)}f$ is right-continuous with left limits in $C_0(\mathbb R^d)$, for all $j \in \{1,\dots,k\}$.

    We now prove that $\mathcal{G}^{\phi_t} f$ in \eqref{eq:Phillips_subLevymultipar} is right-continuous with finite left limits, and there are only a finite number of discontinuity points.
    By definition of the finite measure $\nu_F(t,d\ss)$, we have $\nu_F(t,d\ss) = (\lVert \ss \rVert \wedge 1)\nu(t,d\ss)$, therefore
    \begin{equation*}
        \int_{\RR_+^k \setminus \{\boldsymbol{0}\}}(T_{\ss} f - f) \nu(t,d\ss)
        =
        \int_{\RR_+^k \setminus \{\boldsymbol{0}\}} \frac{T_{\ss} f - f}{\lVert \ss \rVert \wedge 1} \nu_F(t,d\ss).
    \end{equation*}
    The function $\frac{T_{\ss} f - f}{\lVert \ss \rVert \wedge 1}$ is continuous in $\ss$.
    To show that $\frac{T_{\ss} f - f}{\lVert \ss \rVert \wedge 1}$ is bounded, we first show that $\lVert T_{\ss} f - f \rVert \leq \min(2 \lVert f \rVert, \sum_{j=1}^k s_j\left\rVert \mathcal{G}^{(j)} f \right\rVert)$.
   Since $T_{\ss}$ is a contraction semigroup, we have 
    \begin{equation*}
        \lVert T_{\ss} f - f \rVert \leq \lVert T_{\ss} f \rVert + \lVert f \rVert \leq 2 \lVert f \rVert,
    \end{equation*}
    and
    \begin{equation}\label{eq:bound}
         \lVert T_{\ss} f - f \rVert 
        \leq \sum_{j=1}^k s_j\left\rVert \mathcal{G}^{(j)} f \right\rVert.
    \end{equation}
       Equation \eqref{eq:bound} can be proved by induction on $k$, noting that
\begin{equation}
\begin{split}
    \lVert T^{(2)}_{s_2} T^{(1)}_{s_1} f - f\rVert 
    &=
    \left\lVert 
    \int_0^{s_2} T^{(2)}_{r_2} \mathcal{G}^{(2)} T^{(1)}_{s_1}f \, dr_2 
    + 
    \int_0^{s_1} T_{r_1}^{(1)} \mathcal{G}^{(1)} f \, dr_1 
    \right\rVert
    \\
    &\leq
    \int_0^{s_2}
    \left\lVert
    T^{(2)}_{r_2}\mathcal{G}^{(2)}T^{(1)}_{s_1}f
    \right\rVert dr_2
    +
    \int_0^{s_1}
    \left\lVert
    T_{r_1}^{(1)}\mathcal{G}^{(1)}f
    \right\rVert dr_1
    \\
    &\leq 
    \int_0^{s_2} \lVert \mathcal{G}^{(2)}f\rVert \,dr_2 
    + 
    \int_0^{s_1} \lVert \mathcal{G}^{(1)}f\rVert \,dr_1
    \\
    &=
    s_2\lVert \mathcal{G}^{(2)}f\rVert 
    +s_1 \lVert \mathcal{G}^{(1)}f \rVert,
\end{split}
\end{equation}
{where the second inequality follows from the commutativity of the semigroups, which implies $\mathcal{G}^{(2)} T^{(1)}_{s_1} f = T^{(1)}_{s_1} \mathcal{G}^{(2)} f$, together with their contraction property.}
Now, if $\lVert \ss \rVert \geq 1$, the function $\frac{T_{\ss} f - f}{\lVert \ss \rVert \wedge 1} = T_{\ss} f - f$ is bounded.
If $\lVert \ss \rVert < 1$, we have $s_j \leq \max_{j}(s_j) \leq \lVert \ss \rVert < 1$, for every $j \in \{1,\dots,k\}$. 
    Then,
    \begin{equation*}
        \frac{\lVert T_{\ss} f - f \rVert}{\lVert \ss \rVert \wedge 1}
        = 
        \frac{\lVert T_{\ss} f - f \rVert}{\lVert \ss \rVert} 
        \leq
        \sum_{j=1}^k\frac{s_j }{\lVert \ss \rVert} \left\lVert \mathcal{G}^{(j)} f \right\rVert
        \leq
        \sum_{j=1}^k \left\lVert \mathcal{G}^{(j)} f \right\rVert < \infty.
    \end{equation*}
    By Theorem 2 in \cite{nielsen2011weak}, since $\nu_F(t,\cdot)$ weakly converges by assumption, the integral
    \begin{equation*}
        \int_{\RR_+^k \setminus \{\boldsymbol{0}\}}(T_{\ss} f - f) \,\nu(t,d\ss)
    \end{equation*}
    is right-continuous with left limit.
    Therefore, $\mathcal{G}^{\phi_t} f$ in \eqref{eq:Phillips_subLevymultipar} is right-continuous with finite left limits, and there are only a finite number of discontinuity points.
    
    Consider now the partition $\Pi = \{\tau_0,\tau_1,\dots,\tau_n\}$ of the interval $[t_1,t_2] \subset \RR_+$, with $t_1=\tau_0 < \tau_1 < \dots < \tau_n = t_2$.
    Let $|\Pi| := \max_{h=0,\dots,n-1} (\tau_{h+1} - \tau_h)$ and let
    \begin{equation*}
        R_{t_1,t_2}^{\Pi} f := \prod_{h=0}^{n-1} T^{\phi_{\tau_{h}}}_{\tau_{h+1} - \tau_{h}} f, \quad f \in \banach
    \end{equation*}
    and
    \begin{equation} \label{eq:Convolution}
        p_{t_1,t_2}^{\Pi} := \pi_{\tau_1-\tau_0}^{\phi_{\tau_0}} * \pi_{\tau_2-\tau_1}^{\phi_{\tau_1}} * \dots * \pi_{\tau_n-\tau_{n-1}}^{\phi_{\tau_{n-1}}},
    \end{equation}
    where $*$ denotes the convolution product.
    Using expression \eqref{semig_sub_levy}, by Fubini's theorem and by the Markov property of $(T_{\ss})_{\ss \in \RR_+^k}$, we have that, for $f \in \banach$,
    \begin{equation} \label{eq:CommProp}
        T_u^{\phi_{t_1}} T_v^{\phi_{t_2}}f = T_v^{\phi_{t_2}} T_u^{\phi_{t_1}}f.
    \end{equation}

    To prove that the limit of $R^{\Pi}_{t_1,t_2} f$ as $|\Pi| \to 0$ exists for $f \in \banach$, we first need to prove that $D:=\bigcap_{j=1}^k D(\mathcal G^{(j)})$ is dense in $C_0(\mathbb R^d)$ and then apply Theorem~3.1 in \cite{goldstein1969abstract}.
    For $f \in \banach$, define 
    \begin{equation}
        M_{\epsilon}^{(j)} f \coloneqq \frac{1}{\epsilon} \int_0^{\epsilon}T_{s_j}^{(j)}f d s_j,
    \end{equation}
    and let $M_{\epsilon} \coloneqq M_{\epsilon}^{(1)} \cdots M_{\epsilon}^{(k)}.$
    It is well known that $M_{\epsilon}^{(j)} f \in D(\GGG^{(j)})$ and $\lVert M_{\epsilon}^{(j)} f - f \rVert_{\infty} \to 0$ as $\epsilon \to 0$.
    For $f \in \banach$, let $f_{\epsilon} \coloneqq M_{\epsilon} f = M_{\epsilon}^{(1)} \cdots M_{\epsilon}^{(k)} f$.
    Since the marginal semigroups commute, also the operators $M_{\epsilon}^{(j)}$ commute, implying that $f_{\epsilon} \in D(\GGG^{(j)})$, for every $j \in \{1,\dots,k\}$.
    Moreover, since $M_{\epsilon}^{(j)}$ are contractions,
    \begin{equation} \label{eq:ConvEpsilon}
        \lVert f_{\epsilon} - f \rVert_{\infty} \leq \sum_{j=1}^k \lVert M_{\epsilon}^{(j)} f - f \rVert_{\infty} \longrightarrow 0,
    \end{equation}
    as $\epsilon \to 0$.
    
    The fact that $D = \cap_{j=1}^k D(\GGG^{(j)} )$ is dense in $\banach$, the commutative property in \eqref{eq:CommProp}, and the assumptions on $c_j$ and $\nu_F$ imply that the conditions of Theorem 3.1 in \cite{goldstein1969abstract} are verified.
    A first consequence is that the limit of $R^{\Pi}_{t_1,t_2} f$ as $|\Pi| \to 0$ exists for $f \in \banach$, and we denote it by $U_{t_1,t_2}f$.
    Moreover, it is a strongly continuous contraction propagator on $\banach$ and the family of generators of $(U_{t_1,t_2})_{0 \leq t_1 \leq t_2 < \infty}$ is given by
    \begin{equation} \label{eq:U_RightLimit}
        \lim_{h \to 0^+} \frac{U_{t,t+h}f - f}{h} = \mathcal{G}^{\phi_t} f,
    \end{equation}
    for $f \in \cap_{j=1}^k D(\mathcal{G}^{(j)})$.
    
    To prove equations \eqref{eq:generator} and \eqref{eq:CommRightGen} we still have to prove that $U_{t_1,t_2} = T_{t_1,t_2}$ on $\banach$ for every $0 \leq t_1 \leq t_2$.
    For a partition $\Pi$ of $[t_1,t_2]$, from \eqref{semig_sub_levy} and \eqref{eq:Convolution}, we have:
    \begin{equation*}
        R_{t_1,t_2}^{\Pi}f = \prod_{h=0}^{n-1} T^{\phi_{\tau_h}}_{\tau_{h+1} - \tau_{h}} f = \int_{\RR^k_+} T_{\ss} f p_{t_1,t_2}^{\Pi}(d\ss).
    \end{equation*}
    Since $p_{t_1,t_2}^{\Pi}$ is the convolution of laws of Lévy subordinators, the Laplace transform of $p_{t_1,t_2}^{\Pi}$ is
    \begin{equation} \label{eq:LaplaceTransfPartition}
        \int_{\RR^k_+} 
        e^{- \boldsymbol{\lambda} \cdot \uu} p_{t_1,t_2}^{\Pi}(d\uu)
        =
        e^{- \sum_{h=0}^{n-1} \psi(\boldsymbol{\lambda}, \tau_h)(\tau_{h+1}-\tau_h)},
    \end{equation}
    where
    \begin{equation*}
        \psi(\boldsymbol{\lambda},t) = \sum_{j=1}^k \lambda_j c_j(t) 
        +
        \int_{\RR^k_+ \setminus \{\boldsymbol{0}\}}(1-e^{-\boldsymbol{\lambda} \cdot \rr}) \nu(t, d \rr),
        \quad
        \boldsymbol{\lambda} \in \RR^k_+
    \end{equation*}
    is the Laplace exponent of $S^{\phi_{t}}$.
    By hypothesis, $\psi(\boldsymbol{\lambda},t)$ is piecewise continuous in $t$.
    Therefore, as $|\Pi| \to 0$, \eqref{eq:LaplaceTransfPartition} converges to the Laplace transform of the transition probability $\pi_{t_1,t_2}$ of the additive subordinator and $p_{t_1,t_2}^{\Pi}$ weakly converges to $\pi_{t_1,t_2}$.
    For any continuous linear functional $\ell$ on $\banach$ and any $f \in \banach$, as $|\Pi| \to 0$,
    \begin{equation*}
        \ell(R_{t_1,t_2}^{\Pi} f) = 
        \int_{\RR^k_+} \ell(T_{\ss} f) p_{t_1,t_2}^{\Pi}(d\ss)
        \longrightarrow
        \int_{\RR^k_+} \ell(T_{\ss} f) \pi_{t_1,t_2}(d\ss)
        = \ell(T_{t_1,t_2} f),
    \end{equation*}
    since $\ell(T_{\ss} f)$ is a continuous and bounded function in $\ss$.
    Since $R_{t_1,t_2}^{\Pi}f$ converges in norm to $U_{t_1,t_2}f$, it follows that $\ell(U_{t_1,t_2}f)=\ell(T_{t_1,t_2}f)$ for every continuous linear functional $\ell$.
    Hence $U_{t_1,t_2}=T_{t_1,t_2}$.
    From \eqref{eq:U_RightLimit} we have that
    \begin{equation*} 
        \mathcal{G}_tf:=\lim_{h \to 0+} \frac{T_{t,t+h}f - f}{h} = \mathcal{G}^{\phi_t} f, \quad f \in \bigcap_{j=1}^k D(\mathcal{G}^{(j)}).
    \end{equation*}
    It follows that $\cap_{j=1}^k D(\mathcal{G}^{(j)}) \subseteq D(\mathcal{G}_t)$  and from \eqref{eq:Phillips_subLevymultipar} we have \eqref{eq:generator}.
    Finally,~\eqref{eq:CommRightGen} follows from Theorem 3.1(i) in \cite{goldstein1969abstract}.

    We conclude the proof showing that 
  $D \coloneqq \cap_{j=1}^k D(\GGG^{(j)})$ is a core for $\mathcal G_t$ if $D$ is dense in $D(\GGG_t)$ for the graph norm $\lVert f \rVert_{\GGG_t} = \lVert f \rVert_{\infty} + \lVert \GGG_t f \rVert_{\infty}$, for $f \in D(\GGG_t)$ (see Definition II.1.6 in \cite{engel2000one}).
    We need to show that, for $f \in D(\GGG_t)$, there exists $(f_{\epsilon})_{\epsilon > 0} \subset D$ such that $\lVert f_{\epsilon} - f \rVert_{\GGG_t} \to 0$, as $\epsilon \to 0$.
    Consider again $f_{\epsilon} \coloneqq M_{\epsilon} f = M_{\epsilon}^{(1)} \cdots M_{\epsilon}^{(k)} f$. 
    We can show that, for every $j \in \{1, \dots, k\}$, the operators $M^{(j)}_{\epsilon}$ and $T_{t_1,t_2}$ commute, indeed,
    \begin{equation}
    \begin{split}
        M^{(j)}_{\epsilon}T_{t_1,t_2}f &= \frac{1}{\epsilon} \int_0^{\epsilon} T_r^{(j)} \int_{\RR^k_+} T_{\ss} f \, \pi_{t_1,t_2}(d \ss) dr
        \\
        &=
        \int_{\RR^k_+} T_{\ss} \frac{1}{\epsilon} \int_0^{\epsilon} T_r^{(j)} f dr \pi_{t_1,t_2}(d \ss)
        = T_{t_1,t_2} M^{(j)}_{\epsilon} f,
    \end{split}
    \end{equation}
    implying that also $M_{\epsilon}$ and $T_{t_1,t_2}$ commute.
    Therefore, we have
    \begin{equation}
        \GGG_t f_{\epsilon} 
        =
        \mathcal{G}_t M_\epsilon f
        =
        \lim_{h\to 0^+}
        M_\epsilon\frac{T_{t,t+h}f-f}{h}
        =
        M_\epsilon\mathcal{G}_t f.
    \end{equation}
    Thus, since \eqref{eq:ConvEpsilon} holds for every $f \in \banach$, we have, for all $f \in D(\GGG_t)$,
    \begin{equation}
        \lVert f_{\epsilon} - f \rVert_{\GGG_t} = \lVert f_{\epsilon} - f \rVert_{\infty} + \lVert \GGG_tf_\epsilon - \GGG_tf \rVert_{\infty}
        = \lVert M_{\epsilon}f - f \rVert_{\infty} + \lVert M_{\epsilon}\GGG_tf - \GGG_tf \rVert_{\infty}
        \overset{\epsilon \to 0}{\longrightarrow} 0.
    \end{equation}
\end{proof}
Theorem \ref{Li} proves that $T_{t_1,t_2}$ are a strongly continuous  Markov evolution and then they are  a Feller evolution system on $C_0(\mathbb{R}^d)$.

\begin{corollary} \label{cor:generatorind}
    If the multivariate additive subordinator $\SS(t)$ has independent components, then the generator of $\YY(t)$ in \eqref{eq:asp1} reads
    \begin{equation}\label{eq:generatorind}
        \mathcal{G}_t f
        =
        \sum_{j=1}^k {c}_j(t) \mathcal{G}^{(j)} f
        +
        \sum_{j=1}^k\int_{(0,\infty)}(T^{(j)}_{s_j}f-f) \nu_{j}(t,ds_j), 
        \quad 
        f
        \in \bigcap_{j=1}^kD(\mathcal{G}^{(j)}),
    \end{equation}
where $\nu_j(t,ds_j)$, for $j \in \{1,\dots,k\}$, are the time-dependent Lévy measures of the marginal subordinators $S_j(t)$. 
\end{corollary}
\begin{proof}
    By Proposition 5.3 in \cite{cont2003financial}, when $\SS(t)$ has independent components, the Lévy measure of $\SS(t)$ reads, for $A \in \mathcal{B}(\RR_+^k \setminus \{\boldsymbol{0}\})$,
    \begin{equation} \label{eq:SumLevyMeas}
        \nu(t,A) = \sum_{j=1}^k {\nu}_j(t,A_j),
    \end{equation}
    where ${\nu}_j$ is the Lévy measure of the $j$th component $S_j$ and $A_j=\{x_j: x_j\ee_j\in A\}$, $j\in \{1,\ldots, k\}$.
    The proof follows by substituting~\eqref{eq:SumLevyMeas} in~\eqref{eq:generator} and by~\eqref{eq:splitT2}.
\end{proof}
Theorem \ref{teo:sub} contains as a special case the result in \cite{li2016additive}, therefore if the multiparameter semigroup $(T_{\ss})_{\ss\succeq \boldsymbol{0}}$ is associated to a L\'evy process, the process $\YY(t)$ associated to the semigroup $(T_{t_1,t_2})_{t_1<t_2}$ is additive.

We consider the characteristic function $\hat{\mu}_{t_1, t_2}(\xx, \boldsymbol{\xi})$ of the increment of the process $\YY(t_2)-\YY(t_1)$ defined by
\begin{equation}
    \hat{\mu}^{\YY}_{t_1,t_2}(\xx, \boldsymbol{\xi}) 
    = 
    \mathbb{E}[e^{i \xxi \cdot (\YY(t_2)-\YY(t_1))} \vert \YY(t_1) = \xx].
\end{equation}
\begin{proposition} \label{prop:char}
    Let $\YY(t)$ be the process in \eqref{eq:asp1}, the characteristic function $\hat{\mu}_{t_1,t_2}(\xx, \boldsymbol{\xi})$ of an increment $\YY(t_2)-\YY(t_1)$ is given by
       \begin{equation}\label{eq:chinc}
    \hat{\mu}^{\YY}_{t_1,t_2}(\xx, \boldsymbol{\xi})
    =
    \int_{\RR^k_+} \hat{\mu}_{\ss}(\xx, \boldsymbol{\xi}) \pi_{t_1,t_2}(d\ss),
    \end{equation}
    where $\hat{\mu}_{\ss}(\xx, \boldsymbol{\xi})$ is the characteristic function of $\XX(\ss)-\xx$.
\end{proposition}
\begin{proof}
  Taking $f(\xx)=e^{i\xxi\xx}$, the following holds
    \begin{equation*}
    \begin{split}
        &\hat{\mu}^{\YY}_{t_1,t_2}(\xx,\xxi) 
        = 
        e^{-i\xx\xxi}T_{t_1, t_2} e^{i\xx\xxi}\underbrace{=}_{(1)} e^{-i\xx\xxi} \int_{\RR^k_+} T_{\ss}e^{i\xx\xxi}\pi_{t_1,t_2}(d\ss)\\
        &= e^{-i\xx\xxi} \int_{\RR^k_+} \int_{\RR^d}e^{i\yy\xxi}\mu_{\ss}(\xx, d\yy) \pi_{t_1,t_2}(d\ss)\\
        &= \int_{\RR^k_+} \int_{\RR^d}e^{i(\yy-\xx)\xxi}\mu_{\ss}(\xx, d\yy) \pi_{t_1,t_2}(d\ss)=\int_{\RR^k_+} \hat{\mu}_{\ss}(\xx, \boldsymbol{\xi}) \pi_{t_1,t_2}(d\ss),
    \end{split}
    \end{equation*}
    where (1) follows from equation \eqref{eq:DefOpY}.
    \end{proof}

Since $\XX^{(j)}(s_j)$ are Markov processes, the operator semigroups $T^{(j)}_{s_j}$ restricted to $\mathcal{C}^{\infty}_c(\RR^d)$ are pseudo-differential operators as in \eqref{eq:ChTi}, specifically
\begin{equation} \label{eq:PDO_SemigroupMarginalMarkov}
    T^{(j)}_{s_j} f(\xx) 
    =
    (2\pi)^{-d/2}\int_{\RR^d}e^{i\xx \cdot \boldsymbol{\xi}}\hat{\mu}^{(j)}_{s_j}(\xx, \boldsymbol{\xi})\hat{f}(\boldsymbol\xi)d\boldsymbol\xi,
\end{equation}
and, assuming $\mathcal{C}^{\infty}_c(\RR^d)\subset \mathcal{D}(\mathcal{G}^{(j)})$, its
 generator $\mathcal{G}^{(j)}$  has the representation
\begin{equation} \label{eq:PDO_GeneratorMarginalMarkov}
    \mathcal{G}^{(j)}f(\xx)
    =
    -(2\pi)^{-d/2} \int_{\RR^d} e^{i\xx \cdot \boldsymbol{\xi}} q_j(\xx, \boldsymbol{\xi}) \hat{f}(\boldsymbol\xi) d\boldsymbol\xi.
\end{equation}

The following theorem extends Theorem 1.2 in \cite{jacob1998characteristic} and is proved in Theorem~4.5 in~\cite{ruschendorf2016comparison}.
\begin{theorem}\label{thm:symbol}
If $\YY(t)$ is a conservative càdlàg Feller evolution and $C_c^{\infty}(\RR^d)\subset\mathcal{D}(\mathcal{G}_t)$, then $\mathcal{G}_t$ is a pseudo-differential operator on $C_c^{\infty}(\RR^d)$ with symbol $-q(t, \xx, \xxi)$, that is 
\begin{equation} \label{eq:SymbolRightGenerator}
        \mathcal{G}_t f(\xx) 
        =
        -(2\pi)^{-d/2} \int_{\RR^d} e^{i\xx \cdot \boldsymbol{\xi}} q(t,\xx, \boldsymbol{\xi}) \hat{f}(\boldsymbol\xi) d\boldsymbol\xi,
    \end{equation}
The symbol $q$ satisfies the following properties:
\begin{itemize}
\item $q(t,\cdot,\cdot)$ is locally bounded in $\xx, \xxi.$
\item $q(t,\cdot, \xxi)$ is measurable for any $\xxi\in \RR^d.$
\item $q(t,\xx,\cdot)$ is a continuous negative definite function in the sense of Schoenberg for every $\xx\in \RR^d$. 
\end{itemize}
If $q$ is bounded it holds
\begin{equation}
q(t, \xx, \xxi)=-\lim_{h\downarrow 0}
    \frac{
        \hat{\mu}^{\YY}_{t,t+h}(\xx,\xxi)-1
    }{h},
\end{equation}
whenever the limit exists.
\end{theorem}
Since $\hat{\mu}^{\YY}_{t,t}(\xx, \boldsymbol{\xi})=1$, we have the following corollary.
\begin{corollary} Under the assumption of Theorem \ref{thm:symbol}, the symbol $q$ satisfies
    \begin{equation}\label{eq:simbol}
    \begin{split}
        q(t, \xx, \xxi)
        =
        -\left. \frac{d}{dr}\hat{\mu}^{\YY}_{t,r}(\xx,\xxi) \right\vert_{r=t}.
    \end{split}
    \end{equation}  

\end{corollary}

Since
\begin{equation} \label{eq:derivataMuHat}
\begin{split}
    \frac{d}{dt}\hat{\mu}^{\YY}_{s,t}(\xx, \boldsymbol{\xi})&=\lim_{h\rightarrow0}h^{-1}[\hat{\mu}^{\YY}_{s,t+h}(\xx, \boldsymbol{\xi})-\hat{\mu}^{\YY}_{s,t}(\xx, \boldsymbol{\xi})]\\
    &=\lim_{h\rightarrow0}h^{-1}\int_{\RR^d}e^{i \xxi \cdot (\yy-\xx)}[
\hat{\mu}^{\YY}_{t,t+h}(\yy, \boldsymbol{\xi})-\hat{\mu}^{\YY}_{t,t}(\yy, \boldsymbol{\xi})]p_{s,t}(\xx, dy)\\
    &=\int_{\RR^d}e^{i \xxi \cdot (\yy-\xx)}\lim_{h\rightarrow0}h^{-1}[
\hat{\mu}^{\YY}_{t,t+h}(\yy, \boldsymbol{\xi})-1]p_{s,t}(\xx, dy),
    \end{split}
\end{equation}
we have the following equation that links the symbol of the process to the characteristic function of the increments of the process $\YY(t)$:
\begin{equation}
    \frac{d}{dt}\hat{\mu}^{\YY}_{s,t}(\xx,\xxi) = -\int_{\RR^d}e^{i \xxi \cdot (\yy-\xx)}q(t, \yy, \xxi)p_{s,t}(\xx, dy).
\end{equation}

Since $\xxi\rightarrow q(t,\xx,\xxi)$ is continuous and negative definite it enjoys the L\'evy-Khintchine representation:
\begin{equation}
\begin{split}
    q(t,\xx, \boldsymbol{\xi})
    &=
   - i\boldsymbol{\gamma}(t,\xx) \cdot \boldsymbol{\xi} +\frac{1}{2}\boldsymbol{\xi}  \cdot {\Sigma(t,\xx)}\boldsymbol{\xi} -\int_{\RR^d}(e^{i\xxi \cdot \yy}-1-i\xxi \cdot \yy \mathbbm{1}_{B}(\yy)){\nu}_{\YY}(t,\xx,d\yy),
\end{split}
\end{equation}
where $B$ is the unit ball in $\RR^d$ and ${\nu}_{\YY}(t,\xx,d\yy)$ is a time and space dependent L\'evy measure. 
In the following sections, we provide a more explicit expression for the symbol in the Ornstein-Uhlenbeck case.
In Proposition \ref{Prop:symbLev} we present the relevant case of a time-changed L\'evy process for which the symbol can be derived analytically.
\begin{proposition}\label{Prop:symbLev}
    Let $\XX_j(s_j)$ for $j \in \{1,\ldots, k\}$, be independent L\'evy processes. The process $\XX(\ss) = \sum_{j=1}^k \XX_j(s_j)$  is a multiparameter L\'evy process. 
    Let $\SS(t)$ be an additive subordinator with zero drift and differential L\'evy measure  $ \nu(t, \cdot)$ and $\nu_F$ be continuous in $t$, the subordinated  process $\YY(t)=\XX(\SS(t))$ is a time-inhomogeneous additive process. 
    Let $\qq(\xxi) \coloneqq (q_1(\xxi), \ldots, q_k(\xxi))$, where $q_j(\xxi)$ is the symbol of $\XX_j$.
    The Markov evolution $(T_{t_1,t_2})_{t_1<t_2}$ associated to the subordinated process $\YY(t)$ in \eqref{eq:asp1} has a unique generator $\mathcal{G}_t$ with symbol:
    \begin{equation}\label{eq:symb}
        q(t, \xxi) = - \frac{d}{dt}\Psi_t^S(-\qq(\xxi)),
    \end{equation}
    where for any $\ww = (w_1,\dots, w_k) \in \mathbb{C}^k$ with $Re(w_j)\leq 0$, for $j \in \{1,\dots, k\}$, we have
    \begin{equation*}
        \Psi_{t}^S(\ww)
        =
        \boldsymbol{b}(t) \cdot \ww + 
        \int_{\mathbb{R}^k_+} (e^{\ww \cdot \ss}-1) \eta(t,d\ss),
    \end{equation*}
    where $\boldsymbol{b}(t)\in \RR^k_+$ and $\eta(t, d\ss)$ is the time dependent L\'evy measure of $\SS(t)$. 
\end{proposition}
\begin{proof}
    Since $\nu_F$ is continuous and $c_j(t)=0$ for every $j \in \{1,\dots,k\}$, the generator is unique by Theorem \ref{Li}.
    Let $\hat{\mu}^{\YY}_{t}(\boldsymbol{\xi})$ be the characteristic function of $\YY(t)$. Adapting the steps of Theorem 4.7 in \cite{ba} to take into account time inhomogeneity of the  additive  subordinator, we have
    \begin{equation}\label{eq:chinc2}
        \hat{\mu}_{t}(\boldsymbol{\xi})=\exp\{\Psi_t^S(\log\hat{\mmu}(\xxi))\},
    \end{equation}
where $\log\hat{\mmu}(\xxi):=(\log\hat{\mu}^{(1)}(\xxi), \ldots, \log\hat{\mu}^{(k)}(\xxi))$ and $\hat{\mu}^{(j)}(\xxi)$ is the time one characteristic function of $\XX^{(j)}(s_j)$. Since $\XX^{(j)}(\ss_j)$ are L\'evy processes we have $\log\hat{\mu}^{(j)}(\xxi)=-q^{(j)}(\xxi)$ and 
 \begin{equation}\label{eq:chinc3}
    \hat{\mu}_{t}(\boldsymbol{\xi})=\exp\{\Psi_t^S(-\qq(\xxi))\}.
\end{equation}
Then, since $\YY(t)$ has independent increments, 
\begin{eqnarray}
    \hat{\mu}^{\YY}_{s, t}(\boldsymbol{\xi})&=\frac{\hat{\mu}^{\YY}_{t}(\boldsymbol{\xi})}{\hat{\mu}^{\YY}_{s}(\boldsymbol{\xi})}=\exp\{\Psi_t^S(-\qq(\xxi))-\Psi_s^S(-\qq(\xxi))\}.
\end{eqnarray}
We can now compute the symbol as follows
\begin{equation}
\begin{split}
    q(t, \xxi)
    &=
    - \frac{d^+}{dr} \hat{\mu}^{\YY}_{t, r}(\boldsymbol{\xi}) \bigg|_{r=t}
    =
    -\bigg(\frac{d}{dr} \Psi_r^S \big (-\qq(\xxi) \big) \bigg)
    \exp\bigg\{ \Psi_r^S \big(-\qq(\xxi) \big) - \Psi_t^S \big(-\qq(\xxi) \big) \bigg\} \bigg\vert_{r=t}
    \\
    &= -\frac{d}{dt} \Psi_t^S\big(-\qq(\xxi)\big),
    \end{split}
   \end{equation}
which concludes the proof.
\end{proof}
\setcounter{example}{1}
\begin{example}[\bf{continued}] 
    One example of multiparameter Markov process is the multiparameter Brownian motion constructed in \cite{jevtic2017note} as follows.
    Let ${\boldsymbol{B}}^{(j)}(t)$ be independent Brownian motions on $\mathbb{R}^{n_{j}}$ with drift $\boldsymbol{\mu}_j$ and covariance matrix $\boldsymbol{\Sigma}_{j}$.
    Let $\boldsymbol{A}_{j}$ be a $d \times {n_{j}}$ real matrix.
    We can define the process $\boldsymbol{B}_{\boldsymbol{A}} = \{ \boldsymbol{B}_{\boldsymbol{A}} (\boldsymbol{s}),\boldsymbol{s}\in\mathbb{R}^{k}_{+}\}$ as%
    \begin{equation}
        \label{Brho1}
        {\boldsymbol{B}}_{\boldsymbol{A}}(\boldsymbol{s})=\boldsymbol{A}_{1}{{\boldsymbol{B}^{(1)}}%
        }({s}_{1})+\ldots+\boldsymbol{A}_{k}{{\boldsymbol{B}^{(k)}}}({s}_{k}) \,\,\, \boldsymbol{s}%
        \in\mathbb{R}^{k}_{+}.
    \end{equation}
    The process $\boldsymbol{B}_{\boldsymbol{A}}$ is a multiparameter Lévy process on $\mathbb{R}^{d}$, see Example 4.4 in
    \cite{ba}.
    Consider a process $\boldsymbol{Y}$ defined by
    \begin{align} \label{GMPP}
    \boldsymbol{Y}(t):=\boldsymbol{B}_{\boldsymbol{A}}(\boldsymbol{\SS}(t)),%
    \end{align}
    where $\BB_{\AA}(\ss)$ is the multiparameter process in \eqref{Brho1} and $\SS(t)$ is a multivariate additive subordinator independent of $\BB_{\boldsymbol{A}}(\ss)$.
    The process $\YY$ is a particular case of an additive subordinated multiparameter Markov process. 
    Specifically, it is an additive process, with symbol specified later in \eqref{eq:symbAdd}, where $$q^{(j)}(\xxi) = -\log\hat{\mu}_{\AA_j\BB^{(j)}(1)}(\xxi) = -i \xxi \cdot \AA_j \mmu_j + \tfrac{1}{2} \xxi \cdot \AA_j \SSigma_j \AA_j^{\top} \xxi$$  is the symbol of the multivariate Brownian motion $\AA_j\BB^{(j)}(s_j)$.
\end{example}
  
\section{Additive subordination of multiparameter OU processes} \label{sec:M-OUprocess}
\subsection{Multiparameter Ornstein-Uhlenbeck process}

Let $\boldsymbol{W}(t)$ be a $d$-dimensional standard Brownian motion on $\mathbb{R}^{d}$. 
Let $\boldsymbol{K}$ be a real $d\times d$ matrix such that its eigenvalues have positive real parts, $\boldsymbol{\theta}\in\mathbb{R}^{d}$, and let $\La$ be a positive definite $d\times d$ matrix. 
Consider the following $d$-dimensional Ornstein-Uhlenbeck (OU) process solution of 
\begin{equation}
d\boldsymbol{X}(t)=-\boldsymbol{K}\left(  \boldsymbol{X}(t)-\boldsymbol{\theta
}\right)  dt+\La d\boldsymbol{W}(t) , \label{OU}%
\end{equation}
with $\boldsymbol{X}(0)=\xx$. In \eqref{OU}, $\boldsymbol{K}$ is a matrix controlling the mean-reversion rate, $\boldsymbol\theta$ contributes to the non-zero drift (or long-run level), and $\La$, such that $\La\La^{\top}=\Si$, often goes under the name of volatility.

As mentioned in Remark \ref{rem_spaces}, we consider the Banach space $C_0(\mathbb{R}^d)$, endowed with the supremum norm. This choice ensures that the Ornstein--Uhlenbeck semigroup is a strongly continuous contraction semigroup on $C_0(\mathbb{R}^d)$  (for the 1-dimensional case see for instance \cite{li2014time}). 

Let $u_{t}(\xx, \yy)$ be the transition density of the OU process $\XX$, therefore $\mu_{t}(\xx,d\yy) = u_{t}(\xx, \yy)d\yy$. 
Then, the OU semigroup $(T_t)_{t\geq 0}$ defined on the space $C_0(\mathbb{R}^d)$ is such that $T_{t} f(\xx)=\int_{\mathbb R^d}f(\yy)u_{t}(\xx, \yy)d\yy$, where the OU transition kernel has a multivariate normal distribution with the following mean and covariance functions

\[
\boldsymbol{\theta}_{t} = e^{-\KK t} \xx + \left( \boldsymbol{I} - e^{-\KK t} \right) \boldsymbol{\theta}, \quad
{\Si}_{t} = \int_0^{t } e^{-\KK u} \boldsymbol{\Sigma} e^{-\KK^\top u} \, du,
\]
where $\boldsymbol{I}$ is the identity matrix.

The generator of the $d$-dimensional OU process is
\begin{equation}
    \mathcal Gf(\xx)=-\nabla f(\xx) \cdot \KK(\xx-\boldsymbol{\theta})+\frac{1}{2}\mathrm{Tr}\left(\La\La^{\top}\,\nabla^2f(\xx)\right)
\end{equation}
with symbol
\begin{equation}
    q(\xx,\boldsymbol{\xi}) = i\boldsymbol{\xi} \cdot \KK(\xx-\boldsymbol{\theta})+\frac{1}{2}\boldsymbol{\xi} \cdot \Si\boldsymbol{\xi}.
\end{equation}

We want to endow the process with a multiparametric nature in the spirit of Equation \eqref{eq:asp1}, according to the following definition.
A multiparameter product Markov OU (MP-OU) process ${\XX}(\ss;k)$ is an $\RR^{dk}$-valued stochastic process defined by
\begin{equation} \label{eq:MultiOUcomp}
    {\XX}(\ss;k)=( \XX_1(s_1), \ldots, \XX_k(s_k)),
\end{equation}
where $\XX_j(s_j)$ are independent $d$-dimensional OU processes satisfying
\begin{equation}
    d\boldsymbol{X}_j(s_j)=-\boldsymbol{K}_j
    \left( \boldsymbol{X}_j (s_j) - \boldsymbol{\theta}_j \right) ds_j + {\boldsymbol{\Lambda}_j} d\boldsymbol{W}_j(s_j),  
    \quad
    \XX_j(0)=0, 
    \quad 
    j \in \{1, \dots, k\}.\label{terms_OU}%
\end{equation}
Notice that the OU processes in \eqref{terms_OU} have independent driving Brownian motions $\WW_j(s_j)$. 
Moreover, the assumptions $\XX_j(0)= \boldsymbol{0}$, for $j \in \{1, \dots, k\}$, imply $\XX(\boldsymbol{0};k) = \boldsymbol{0}$. 
For the transition kernels and generators we use the same equations with arbitrary initial states $\XX_j(0)=\xx^{(j)}$, $\xx=(\xx^{(1)}, \ldots, \xx^{(k)})\in\RR^{dk}$. The zero initial state remains the choice for the model itself. Each $\KK_j$ has eigenvalues with positive real parts and $\La_j\La_j^\top$ is positive definite.
Let $\XX(\ss;k)$ be an MP-OU process and $\SS(t)$ an independent additive subordinator. The process $\YY(t;k) = \XX(\SS(t);k)$ in \eqref{eq:asp1} is an additive subordinated MP-OU process.

We conclude this section by finding the symbol of the additive subordinated MP-OU process $\YY(t;k)$.
If $q_F(x,\xi)$ is the negative definite Fourier symbol of a Lévy process and $\phi$ is the Bernstein function (Laplace exponent) of a L\'evy subordinator, then the symbol $q$ of the resulting subordinated process is the composition $q(x,\xi)=\phi(q_F(x,\xi))$ (\cite{orsingher2016time}, \cite{Evans2007}).
  In the additive subordination case, 
   $\phi$ is a Bernstein function only for fixed time $t$, giving rise to a time-dependent symbol  $q(t, \xx, \xi)$ whose derivation is non-trivial.
    Proposition \ref{Prop:symbLev}  considers the case of an additive subordinated L\'evy process, while Proposition \ref{prop:triplet} addresses the corresponding problem for the OU process.
    \begin{lemma}
        
For any $\ss\in \RR^k_+$ there is a constant $C_{\xx}$ such that
\begin{align}
    |\hat\mu_{\ss}(\xx,\boldsymbol\xi)-1|
    &\leq C_{\xx}(1+|\boldsymbol\xi|)^2(1\wedge|\ss|),
    \label{eq:OU_char_bound}\\
    \int_{\RR^{dk}}(1\wedge|\yy|^2)
    u_{\ss}(\xx,\xx+d\yy)
    &\leq C_{\xx}(1\wedge|\ss|),
    \label{eq:OU_jump_bound}\\
    \left|
    \int_{\RR^{dk}}\yy\,\mathbbm{1}_B(\yy)
    u_{\ss}(\xx,\xx+d\yy)
    \right|
    &\leq C_{\xx}(1\wedge|\ss|).
    \label{eq:OU_drift_bound}
    \end{align}
   
\end{lemma}
\begin{proof}
 For fixed $\xx$, set $Z_{\ss}:=\XX(\ss;k)-\xx$ under
$\mathbb P_{\xx}$. The OU mean and covariance formulas
(see \cite[Section~4.4.6]{gardiner1985handbook}) imply
\[
    |\mathbb E_{\xx}Z_{\ss}|+\mathbb E_{\xx}|Z_{\ss}|^2
    \leq C_{\xx}|\ss|,
    \qquad |\ss|\leq1.
\]
Using Taylor's inequality and
\[
    |\mathbb E_{\xx}[Z_{\ss}\mathbbm{1}_B(Z_{\ss})]|
    \leq |\mathbb E_{\xx}Z_{\ss}|+\mathbb E_{\xx}|Z_{\ss}|^2,
\]
where $B=\{\yy:|\yy|\leq1\}$, therefore for $|\ss|<1$ we obtain the bounds in equation \eqref{eq:OU_char_bound}, \eqref{eq:OU_jump_bound}, and \eqref{eq:OU_drift_bound}.
These estimates hold for all $\ss\in\RR_+^k$, since their
left-hand sides are smaller than $2$, $1$, and $1$, respectively, when $|\ss|>1$.
For the one-dimensional versions of the last two estimates, see also \cite[proof of Theorem~1]{li2013ornstein}.
    
    \end{proof}
\begin{proposition} \label{prop:CharMP-OU}
    Given an additive subordinated MP-OU process $\YY(t;k) = \XX(\SS(t);k)$, where $\XX(\ss;k)$ is given in \eqref{eq:MultiOUcomp}, the characteristic function of the increment $\YY(t_2;k) - \YY(t_1;k)$ is given by, for $\xx, \xxi \in \RR^{dk}$,
    \begin{equation} \label{eq:chinc2}
    \hat{\mu}^{\YY(\cdot;k)}_{t_1,t_2}(\xx, \boldsymbol{\xi})
    =
    \int_{\RR^k_+} \prod_{j=1}^k \hat{\mu}^j_{s_j}(\xx^{(j)}, \boldsymbol{\xi}^{(j)}) \pi_{t_1,t_2}(d\ss),
    \end{equation}
    where $\hat{\mu}^j_{s_j}(\xx^{(j)}, \boldsymbol{\xi}^{(j)})$ is the characteristic function of $\XX_j(s_j)-\xx^{(j)}$, for every $j \in \{1,\dots,k\}$, and $\xx^{(j)}, \xxi^{(j)} \in \RR^d$ are such that $\xx = (\xx^{(1)}, \dots, \xx^{(k)})$ and $\xxi = (\xxi^{(1)}, \dots, \xxi^{(k)})$.
\end{proposition}
\begin{proof}
    By Proposition \ref{prop:char},
    \begin{equation}
        \hat{\mu}^{\YY(\cdot;k)}_{t_1,t_2}(\xx, \boldsymbol{\xi})
        =
        \int_{\RR^k_+} \hat{\mu}_{\ss}(\xx, \boldsymbol{\xi}) \pi_{t_1,t_2}(d\ss).
    \end{equation}
    By construction,
    \begin{equation}
    \begin{split}
        \hat{\mu}_{\ss}(\xx, \boldsymbol{\xi}) 
        &=
        \mathbb{E}_{\xx}[ e^{i \xxi \cdot (\XX(\ss) - \xx)} ]
        =
        \prod_{j=1}^k \mathbb{E}_{\xx^{(j)}}[ e^{i \xxi^{(j)} \cdot (\XX_j(s_j) - \xx^{(j)})} ]
        =
        \prod_{j=1}^k \hat{\mu}^j_{s_j}(\xx^{(j)}, \boldsymbol{\xi}^{(j)})
    \end{split}
    \end{equation}
\end{proof}

In the following proposition, $u_{\ss}(\xx,\xx+d\yy)$ is the law of the increment.
\begin{proposition} \label{prop:triplet}
    Under the assumptions of Theorem \ref{Li}, the additive subordinated 
       MP-OU process $\YY(t;k) = \XX(\SS(t);k)$, where $\XX(\ss;k)$ is given in \eqref{eq:MultiOUcomp},
       has symbol $q(t,\xx, \boldsymbol{\xi}) $ given by 
    \begin{equation}
        q(t,\xx, \boldsymbol{\xi}) =
        -i \boldsymbol{\gamma}_{\YY}(t,\xx) \cdot \boldsymbol{\xi} 
        +
        \frac{1}{2} \xxi \cdot \Sigma_{\YY}(t) \xxi 
        -
        \int_{\RR^{dk}} \big(e^{i\xxi \cdot \yy} - 1 - i \xxi \cdot \yy \mathbbm{1}_{B}(\yy) \big){\nu}_{\YY}(t,\xx,d\yy),
    \end{equation}
    where
    \begin{align}
        \boldsymbol{\gamma}_{\YY}(t,\xx)
        &=
        \sum_{j=1}^k c_j(t) {\tilde{\KK}_j} ( {\boldsymbol{\theta}} -\xx) + \int_{\RR^k_+ \setminus \{\boldsymbol{0}\}} \int_{\RR^{dk}} \yy \mathbbm{1}_B(\yy) {u_{\ss}(\xx, \xx + d\yy)} \, \nu(t, d\ss)
        \\ 
        \Sigma_{\YY}(t)
        &=
        \sum_{j=1}^k c_j(t) \tilde{\boldsymbol{\Sigma}}_{j}
        \\
        {\nu}_{\YY}(t,\xx,d\yy)&=\int_{\RR^k_+ \setminus \{\boldsymbol{0}\}}(u_{\ss}(\xx, \xx + d\yy))\nu(t, d\ss),
    \end{align}
       where $\tilde{\boldsymbol{\Sigma}}_j = \tilde{\La}_j\tilde{\La}_j^{\top} = diag(\boldsymbol{0},\dots,\boldsymbol{\Sigma}_j,\dots,\boldsymbol{0})$, and $\tilde{\La}_j$ and $\tilde{\boldsymbol{K}}_j$ are $dk \times dk$ matrices built as block-diagonal matrices,
    \begin{equation*}
        \tilde{\boldsymbol{K}}_j = \diag (\boldsymbol{0},\dots,\KK_j,\dots,\boldsymbol{0})
        \quad \text{and} \quad
        \tilde{\La}_j = \diag (\boldsymbol{0},\dots,\La_j,\dots,\boldsymbol{0}).
    \end{equation*}
    Finally,
      ${\nu}_{\YY}(t,\xx,d\yy)$ is a time and space dependent Lévy measure and ${\nu}_{\YY}(t,\xx,\RR^{dk}\setminus\{\boldsymbol{0}\})<\infty$ if and only if ${\nu}( t, \RR^{k}\setminus\{\boldsymbol{0}\})<\infty$. 
\end{proposition}
\begin{proof}
    From Theorem \ref{Li}, the generator $\GGG_t$ associated to the Markov process $\YY(t;k)$ is given by
    \begin{equation}
        \GGG_t f
        =
        \sum_{j=1}^k {c}_j(t) {\mathcal{G}}^{(j)} f
        +
        \int_{\RR^k_+ \setminus \{\boldsymbol{0}\} }({T}_{\ss}f-f)\nu(t,d\ss), 
        \quad 
        f
        \in \bigcap_{j=1}^kD(\mathcal{G}^{(j)}).
    \end{equation}
    For every $j \in \{1,\dots,k\}$, the operator $\GGG^{(j)}$ is the generator associated with the one-parameter Markov process $\XX^{(j)}(s_j;k)$, where $\XX^{(j)}(s_j;k) = \XX(s_j \ee_j;k)$.
       From \eqref{eq:PDO_GeneratorMarginalMarkov}, it holds
    \begin{equation} \label{eq:psudo_proof}
        \GGG^{(j)} f(\xx)
        =
        -(2\pi)^{-dk/2} \int_{\RR^{dk}} e^{i\xx\boldsymbol{\xi}} q_j(\xx, \boldsymbol{\xi}) \hat{f}(\boldsymbol\xi) d\boldsymbol\xi,\quad f \in {\mathcal{C}^{\infty}_c(\RR^{dk})},
    \end{equation}
    where $q_j(\xx, \boldsymbol{\xi})$ is the symbol of
    $\XX^{(j)}(s_j;k)$.
    By the construction of $\XX(\ss;k)$ in \eqref{eq:MultiOUcomp}, it follows that $\XX^{(j)}(s_j;k) = \XX(s_j \ee_j;k) = (\XX_1(0), \dots, \XX_j(s_j), \dots , \XX_k(0))$.
    Therefore, the process $\XX^{(j)}(s_j;k)$ is an $\RR^{dk}$-valued OU process that satisfies the following SDE,
    \begin{equation}
        d\XX^{(j)}(s_j;k) = -\tilde{\boldsymbol{K}}_j
        \left( \XX^{(j)}(s_j;k) - \boldsymbol{\theta} \right) ds_j + \tilde{\La}_j d\boldsymbol{W}(s_j),  
        \quad
        \XX^{(j)}(0;k) = \boldsymbol{0},
    \end{equation}
    where $\boldsymbol{\theta} = (\boldsymbol{\theta}_1,\dots,\boldsymbol{\theta}_k)$, $\boldsymbol{W} = (\boldsymbol{W}_1, \dots, \boldsymbol{W}_k)$, and where $\tilde{\boldsymbol{K}}_j$ and $\tilde{\La}_j$ are $dk \times dk$ matrices built as block-diagonal matrices,
    \begin{equation*}
        \tilde{\boldsymbol{K}}_j = \diag (\boldsymbol{0},\dots,\KK_j,\dots,\boldsymbol{0})
        \quad \text{and} \quad
        \tilde{\La}_j = \diag (\boldsymbol{0},\dots,\La_j,\dots,\boldsymbol{0}).
    \end{equation*}
    Let $\tilde{\boldsymbol{\Sigma}}_j = \tilde{\La}_j\tilde{\La}_j^{\top} = diag(\boldsymbol{0},\dots,\boldsymbol{\Sigma}_j,\dots,\boldsymbol{0})$, it follows that the symbol $q_j(\xx,\xxi)$ in \eqref{eq:psudo_proof} is given by
    \begin{equation*}
        q_j(\xx,\xxi) = -i \xxi \cdot \tilde{\KK}_j (\boldsymbol{\theta} - \xx)+ \frac{1}{2}\xxi \cdot \tilde{\boldsymbol{\Sigma}}_j \xxi
        =
        -i \xxi^{(j)} \cdot \KK_j (\boldsymbol{\theta}_j - \xx^{(j)})+ \frac{1}{2}\xxi^{(j)} \cdot \boldsymbol{\Sigma}_j \xxi^{(j)},
    \end{equation*}
    where $\xx^{(j)},\xxi^{(j)} \in \RR^d$ are such that $\xx = (\xx^{(1)}, \dots, \xx^{(k)})$ and $\xxi = (\xxi^{(1)}, \dots, \xxi^{(k)})$.
    From \eqref{eq:PDO_SemigroupMarginalMarkov}, we have
    \begin{equation*}
        \int_{\RR^k_+ \setminus \{\boldsymbol{0}\}}(T_{\ss}f(\xx)-f(\xx))\nu(t, d\ss) 
        =(2\pi)^{-dk/2} \int_{\RR^{dk}} e^{i\xx \cdot \xxi} \int_{\RR^k_+ \setminus \{\boldsymbol{0}\}} \big( \hat{\mu}_{\ss}(\xx, \boldsymbol{\xi})-1 \big) \nu(t,d\ss) \hat{f}(\boldsymbol\xi) d\xxi,
    \end{equation*} 
    where the change of integration is justified by~\eqref{eq:OU_char_bound}.

    We have
    \begin{equation*}
    \begin{split}
        &\int_{\RR^k_+ \setminus \{\boldsymbol{0}\}} (\hat{\mu}_{\ss}(\xx, \boldsymbol{\xi})-1) \nu(t, d\ss) =
        \int_{\RR^k_+ \setminus \{\boldsymbol{0}\}} \left( \int_{\RR^{dk}}(e^{i\yy \cdot \xxi}-1) u_{\ss}(\xx, \xx + d\yy) \right) \nu(t, d\ss)
        \\
        &= \int_{\RR^k_+ \setminus \{\boldsymbol{0}\}}\int_{\RR^{dk}}(e^{i \yy \cdot \xxi}-1-i \yy \cdot \xxi\boldsymbol{1}_B(\yy) ) u_{\ss}(\xx, \xx + d\yy) \nu(t, d\ss)
        \\
        &\qquad + \int_{\RR^k_+ \setminus \{\boldsymbol{0}\}}\int_{\RR^{dk}}i \yy  \cdot \xxi \boldsymbol{1}_B(\yy) u_{\ss}(\xx, \xx + d\yy) \nu(t, d\ss)
         \\
        &\overset{*}{=} \int_{\RR^{dk}}(e^{i \yy \cdot \xxi} -1-i \yy \cdot \xxi \boldsymbol{1}_B(\yy ) )\int_{\RR^k_+ \setminus \{\boldsymbol{0}\}} u_{\ss}(\xx, \xx + d\yy) \nu(t,d\ss)
        \\
        &\qquad + \int_{\RR^k_+ \setminus \{\boldsymbol{0}\}}\int_{\RR^{dk}} i\yy \cdot \xxi \boldsymbol{1}_B(\yy) u_{\ss}(\xx, \xx + d\yy) \nu(t, d\ss)
        \\
        &= \int_{\RR^{dk}}(e^{i\yy \cdot \xxi}-1-i\yy \cdot \xxi\boldsymbol{1}_B(\yy ) )\nu_{\YY}(t,\xx,d\yy)\\
        &\qquad 
        + \int_{\RR^k_+ \setminus \{\boldsymbol{0}\}}\int_{\RR^{dk}} i \yy \cdot \xxi\boldsymbol{1}_B(\yy) u_{\ss}(\xx, \xx + d\yy) \nu(t, d\ss).
    \end{split}
    \end{equation*} 
    The change of integration order in $\overset{*}{=}$ is justified since ${\nu}_{\YY}(t,\xx,d\yy) = \int_{\RR^k_+} u_{\ss}(\xx, \xx + d\yy) \, \nu(t, d\ss)$ is a L\'evy  measure, i.e.,
    \begin{equation}\label{eq:levmeas}
        \int_{\RR^k_+ \setminus \{\boldsymbol{0}\}}\int_{\RR^{dk}}(1\wedge |\yy|^2) u_{\ss}(\xx, \xx + d\yy) \, \nu(t, d\ss)<\infty.
    \end{equation}  

    Integrating these bounds in \eqref{eq:OU_jump_bound} and \eqref{eq:OU_drift_bound} with respect to $\nu(t,d\ss)$ proves~\eqref{eq:levmeas} and
    \begin{equation}
        \int_{\RR^k_+\setminus\{\boldsymbol{0}\}}
        \left\lvert
        \int_{\RR^{dk}}\yy\,\mathbbm{1}_B(\yy) u_{\ss}(\xx,\xx+d\yy)
        \right\rvert
        \nu(t,d\ss)<\infty,
    \end{equation}
    so that the drift correction is well defined.
 
    As a consequence ${\nu}_{\YY}(t,\xx,\RR^{dk}\setminus \{\boldsymbol{0}\})<\infty$ if and only if ${\nu}( t, \RR^{k}\setminus \{\boldsymbol{0}\})<\infty$ (see Remark~4.3 in~\cite{li2016additive}).  
    Putting everything together
    \begin{equation}
    \begin{split}
        \mathcal{G}_t f
        &= -(2\pi)^{-dk/2}\int_{\RR^{dk}}e^{i\xx \cdot \xxi} \left[
        \sum_{j=1}^k c_j(t) q_j(\xx, \xxi)\right.\\
        &\left.- \int_{\RR^k_+ \setminus \{\boldsymbol{0}\}}\int_{\RR^{dk}} i \yy \cdot \xxi\boldsymbol{1}_B(\yy) u_{\ss}(\xx, \xx + d\yy) \, \nu(t, d\ss)\right. \\
        &\left.- \int_{\RR^{dk}}(e^{i\yy \cdot \xxi}-1-i\yy \cdot \xxi \boldsymbol{1}_B(\yy ) )\nu_{\YY}(t,\xx,d\yy) 
    \right]
        \hat{f}(\xxi)d\xxi,
    \end{split}
    \end{equation}
    and the statement follows naturally.
\end{proof}

In the next section, we consider a specific transformation of $\YY(t;k) = \XX(\SS(t);k)$ to provide an example of interest in financial applications.  

\section{An application to financial modeling}\label{sec:fin}
In this section, we provide a constructive example of a linear transformation of the additive subordinated MP-OU process $\YY(t;k)$ based on the following goals:
\begin{enumerate}
    \item Model mean reversion $\Rightarrow$  multiparameter OU
    \item Allow each asset to evolve according to its own clock $\Rightarrow$ allow for {\it independence} among one-dimensional marginals evolving on different time scales;
    \item Ensure {\it parsimony} in the number of parameters $\Rightarrow$ we carefully select an appropriate family of additive subordinators;
    \item Maintain analytical {\it tractability} and accommodate {\it time-inhomogeneity} $\Rightarrow$ we adopt a dependence structure that is kept as simple as possible.
\end{enumerate}

\subsection{The multiparameter OU process}
Let us define a $(d+1)$-parameter process $\XX(\ss)$ on $\RR^d$ as follows:
\begin{equation}\label{eq:MultiOUFB}
    \XX(\ss) =
    \begin{pmatrix}
        U_1(s_1) +  a_1U_{d+1}(s_{d+1})
        \\
        \vdots\\
        U_d(s_d) +a_d U_{d+1}(s_{d+1})
    \end{pmatrix},
    \quad
    \ss \in \RR_+^{d+1},
\end{equation}
where $a_j \in \RR_+$, and where $U_j$, for $j \in \{1,\ldots, d\}$, are independent one-dimensional OU processes, defined as unique strong solutions of the following SDEs
\begin{equation} \label{eq:SDE_Uj}
    \begin{cases}
        \mathrm{d} U_j(t) = -k_j(U_j(t)-\theta_j) \mathrm{d}t + \sigma_j \mathrm{d}W_j(t),
        \\
        U_j(0) = 0,
    \end{cases}
    \quad j \in \{1,\dots,d\},
\end{equation}
where $k_j, \sigma_j \in \RR_+$ and $\theta_j \in \RR$, for $j \in \{1,\dots,d\}$.
The process $U_{d+1}$ in \eqref{eq:MultiOUFB} is a one-dimensional OU process, independent of $U_1,\dots,U_d$, defined as the unique strong solution of the following SDEs
\begin{equation}\label{eq:SDE_U}
\begin{cases}
    dU_{d+1}(t) = - U_{d+1}(t) dt + dW_{d+1}(t),
    \\
    U_{d+1}(0) = 0.
\end{cases}
\end{equation}

Since the characteristic exponent $\psi_t(\xi)$ of a one dimensional  OU process $V(t)$ with parameters $(k_V, \theta_V, \sigma_V)$ (\cite{gardiner1985handbook}, page 75), $V(0)=0$, is
\begin{equation}
   \psi_t(\xi)=i\theta_V(1-e^{-k_Vt})\xi-\frac{\sigma_V^2}{4k_V}(1-e^{-2k_Vt})\xi^2,
\end{equation}
the characteristic exponent of $\XX(\ss)$ is
\begin{equation}\label{eq:expOU}
\begin{split}
   \psi_{\ss}(\xxi)
   &=
   \sum_{j=1}^d
   \psi_{U_j(s_j)}(\xi_j)+\psi_{U_{d+1}(s_{d+1})} \bigg(\sum_{h=1}^da_h\xi_h \bigg)
   \\
   &=  
   \sum_{j=1}^d
   \bigg\{ 
   i \theta_j(1-e^{-k_js_j})\xi_j - \frac{\sigma_j^2}{4k_j}(1-e^{-2 k_j s_j})\xi_j^2
   \bigg\} 
   - \frac{1}{4} \big( 1-e^{-2s_{d+1}} \big)
   \left(\sum_{h=1}^d\xi_ha_h\right)^2.
   \end{split}
\end{equation}

Let  $q_j(x_j, \xi_j)$ be the symbol of $U_j(s_j)$, for every $j \in \{1,\dots, d\}$, and $q_{d+1}(x_{d+1}, \xi_{d+1})$ be the symbol of $U_{d+1}$. We have
\begin{equation}
    q_j(x_j, \xi_j)=-ik_j(\theta_j-x_j)\xi_j+\frac{1}{2}\sigma^2_j\xi^2_j,
\end{equation}
and
\begin{equation}
       q_{d+1}(x_{d+1}, \xi_{d+1}) = +i x_{d+1}\xi_{d+1}+\frac{1}{2} \xi^2_{d+1}.
  \end{equation}

\begin{proposition} \label{prop:ScaledMargProcesses}
    Let $\VV(t)=\XX(\frac{t}{k_1},\ldots,\frac{t}{k_d},t)$, where $k_j > 0$ are those of \eqref{eq:SDE_Uj}.
    Then, $\VV(t)$ is a $d$-dimensional OU process satisfying the following SDE
    \begin{equation*}
    \begin{cases}
        d \VV(t) = -\boldsymbol{I} (\VV(t)-\boldsymbol{\theta}) dt + \La \, d\WW(t),
        \\
        \VV(0) = \boldsymbol{0},
    \end{cases}
    \end{equation*}
    where $\boldsymbol{I}$ is the identity matrix, $\boldsymbol{\theta} = (\theta_1,\ldots,\theta_d)$, and $\Si = \La\La^{\top} = (\tilde{\sigma}_{j,h})$, where $\tilde{\sigma}_{j,j}=\frac{\sigma_j^2}{k_j} + a_j^2 $ and $\tilde{\sigma}_{j,h} = a_j a_h$, for $j,h \in \{1,\dots,d\}$, $h \neq j$.
    Furthermore, for every $j \in \{1,\dots,d\}$, the one-dimensional marginal process $V_j(t)$ is an OU process that satisfies the SDE
    \begin{equation}\label{eq:margin}
    \begin{cases}
        \mathrm{d} V_j(t) = -(V_j(t)-\theta_j) \mathrm{d}t + \sigma^V_j \mathrm{d} \widetilde{W}_j(t),
        \\
        V_j(0) = 0,
    \end{cases}
    \end{equation}
    where $\sigma^V_j= \sqrt{\frac{\sigma_j^2}{k_j} + a_j^2 } $.
\end{proposition}

\begin{proof}
A one dimensional OU process $Z(t)$ with parameters $(k_Z, \theta_Z, \sigma_Z)$ is a Gauss-Markov process and can be written as (\cite{di2001computational}) 
\begin{equation} \label{transf_OU_BM}
    Z(t)=\theta_Z(1-e^{-k_Z t})+\sigma_Z W\left(\frac{1-e^{-2k_Zt}}{2k_Z}\right),
\end{equation}
where we have used that $Z(0)=0$. Let $V_j(t) = X_j(t/k_1,\ldots, t/k_d,t)=U_j(t/k_j)+a_j U_{d+1}(t)$.
From \eqref{transf_OU_BM}, we have:
\begin{equation*}
\begin{split}
    V_j (t)& = \theta_j(1-e^{-{t}})+\sigma_j  W_j \left( \frac{1-e^{-2t}}{2k_j} \right) + a_j  W\left( \frac{1-e^{-2t}}{2} \right)\\
   & = 
    \theta_j(1-e^{-t})+\sqrt{\frac{\sigma_j^2}{k_j} + a_j^2} \,\,  \widetilde{W}_{j}\left( \frac{1-e^{-2t}}{2} \right), 
    \end{split}
\end{equation*}
where we used the independence between $W(t)$ and $W_j(t)$ and the scaling properties of the Brownian motion and where $\widetilde{W}_{j}$, $j=1,\ldots, d$, are standard non-independent Brownian motions. 
Therefore $V_j(t)$ satisfies Equation \eqref{eq:margin}.
Furthermore, for $j,h \in\{1,\dots,d\}$, $j \neq h$,
\begin{equation*}
\begin{split}
    (\Sigma_t)_{j,h} = \mathrm{Cov}(V_j (t), V_h(t))
    &= 
    \mathrm{Cov} \left( a_j  W\left( \frac{1-e^{-2t}}{2} \right), a_h  W\left( \frac{1-e^{-2t}}{2} \right) \right)
    \\
    &= 
    a_j a_h  \left( \frac{1-e^{-2t}}{2} \right) = \tilde{\sigma}_{j,h}\left( \frac{1-e^{-2t}}{2} \right),
    \end{split}
    \end{equation*}
    and
    \begin{equation*}
        (\Sigma_t)_{j,j} = \left( \frac{\sigma_j^2}{k_j} + a_j^2 \right) \left( \frac{1-e^{-2t}}{2} \right).
    \end{equation*}
\end{proof}

\begin{remark}
   The dependence structure of the process $\XX(\ss)$ in \eqref{eq:MultiOUFB} can be weakened by considering $\aa \odot \tilde{\UU}$, where $\boldsymbol{a}\odot\tilde{\UU}=(a_1\tilde{U}_1,\ldots, a_d\tilde{U}_d)$   is the Hadamard product, instead of the process $(a_1U_{d+1},\dots, a_d U_{d+1})$, where $\tilde{\UU}=(\tilde{U}_1,\ldots, \tilde{U}_d)$ is a multivariate Ornstein-Uhlenbeck process with covariance function $\Sigma=(\sigma_{jk})$ with $\sigma_{jj}=1$, for $j \in \{1,\ldots, d\}$. In this case, the characteristic exponent in \eqref{eq:expOU} becomes
\begin{equation}\label{eq:expOU2}
\begin{split}
   \psi_{\ss}(\xxi)
   &=
   \sum_{j=1}^d
   \psi_{U_j(s_j)}(\xi_j)+\psi_{\tilde{\UU}(s_{d+1})} (\boldsymbol{a}\odot\xxi).
   \end{split}
\end{equation}
By our choice of the parameters, Proposition \ref{prop:ScaledMargProcesses} still holds. 
\end{remark}

\subsection{The Sato subordinator}

Sato processes are additive processes associated to self-decomposable distributions, \cite{sato1991self}.
In particular, if $\ZZ(t)$ is a $\rho$-Sato process, then
\begin{equation*}
\boldsymbol{Z}(t)\overset{\mathcal{L}}{=}t^{\rho}\boldsymbol{Z}, 
\end{equation*}
where $\ZZ=\ZZ(1)$ is a zero drift self-decomposable random variable and $\rho>0$ is the self-similar exponent (see   \cite{carr2007self,eberlein2009sato,li2016additive} for a thorough exposition).
A $k$-dimensional Sato subordinator is an $\RR^k$-valued Sato process with positive and non-decreasing one-dimensional marginal trajectories.

We study multivariate Sato subordinators $\SS(t)$ such that $\SS(1)$ has distribution in the class of  tempered stable ($T\alpha S$) distributions introduced in \cite{rosinski2007tempering}. 
Tempered stable distributions are obtained by tempering the radial component of the L\'evy measure (in polar coordinates) of an $\alpha$-stable distribution, $\alpha \in (0,2)$.
It is well known that the L\'evy measure $\nu_0$ of an $\alpha$-stable distribution, $\alpha\in (0,2)$, is of the form
\begin{equation}\label{nustable}
    \nu_0(E) = \int_{\mathcal{S}^{d-1}} \int_{\RR_+} \boldsymbol{1}_E(r \w)\frac{1}{r^{\alpha+1}}\mathrm{d}r \, \lambda(\mathrm{d}\w),
   \end{equation}
where $E\subset \RR^d$ is a Borel set, $\lambda$ is a finite measure on the unit $(d-1)$-sphere $\mathcal{S}^{d-1} \subseteq \RR^d$.
Tempered stable distributions are formally defined as follows (\cite{rosinski2007tempering}).
\begin{definition}\label{tstable}
A probability measure $\mu$ on $\RR^d$ is called tempered $\alpha$-stable (T$\alpha$S)
if it is infinitely divisible without Gaussian part and it has L\'evy measure $\nu$ of the following form
\begin{equation}\label{nuAstable}
    \nu(dr, d\w)=\frac{b(r, \w)}{r^{\alpha+1}}\mathrm{d}r \,\lambda(\mathrm{d}\w),
\end{equation}
where $\alpha\in (0,2)$, $\lambda$ is a finite measure on $\mathcal{S}^{d-1}$, and $b:(0,\infty) \times \mathcal{S}^{d-1} \rightarrow (0,\infty)$ is a Borel function such that
$b(\cdot,\w)$ is completely monotone with $\lim_{r\rightarrow\infty}b(r, \w) = 0$, for every $\w\in \mathcal{S}^{d-1}$.
The measure $\mu$ is called a proper T$\alpha$S distribution if, in addition to the above, $b(0+, \w) = 1$ for each $\w\in \mathcal{S}^{d-1}$.
\end{definition}

\setcounter{example}{1}
\begin{example}[\bf{continued}] 
Let the subordinator  $\SS(t)$ in \eqref{GMPP} be a Sato process. By Proposition \ref{Prop:symbLev} the process $\YY(t)$ in \eqref{GMPP} has symbol given by

\begin{equation}\label{eq:symbAdd}
        q(t, \xxi) 
        = 
        - \frac{d}{dt}\Psi_t^S(-\qq(\xxi))
        =
         \rho t^{\rho-1}\sum_{j=1}^kq^{(j)}(\xxi) (\partial_j \Psi^S_1)(-t^\rho\qq(\xxi)),
    \end{equation}
        where $\partial_j\Psi^S_1(x)$ is the partial derivative of $\Psi^S_1(x)$ and $q^{(j)}(\xxi)=-\log\hat{\mu}_{\AA_j\BB_j(1)}(\xxi)$ is the symbol of the multivariate Brownian motion $\AA_j\BB_j(s_j).$
\end{example}

We consider tempered stable distributions with an exponential tempering function $b(r, \w)=e^{-\beta( \w)r}$ (see \cite{semeraro2022multivariate}).
\begin{definition}
An infinitely divisible  probability measure $\mu$ on $\RR^d$ is called  exponential T$\alpha$S ($\ets$) if it is without Gaussian part and it has L\'evy measure $\nu$ of the form
\begin{equation}\label{nuTstable}
    \nu(E) = 
    \int_{\mathcal{S}^{d-1}} \int_{\RR_+} \boldsymbol{1}_E(r \w) \frac{e^{-\beta( \w)r}}{r^{\alpha+1}}\mathrm{d}r \, \lambda(\mathrm{d}\w),
    \end{equation}
where  $E\subset \RR^d$ is a Borel set, $\alpha\in(0, 2),$  $\lambda$ is a finite measure on $\mathcal{S}^{d-1}$ and $\beta: \mathcal{S}^{d-1}\rightarrow (0, \infty)$ is a Borel-measurable function.
\end{definition}
In \cite{semeraro2022multivariate}, it is proved that the L\'evy measure in \eqref{nuTstable}  is the L\'evy measure associated to an $\ets$ Sato subordinator $\SS(t)$ if and only if
$\alpha\in (0,1)$  and $\lambda$ has support on $\mathcal{S}^{d-1}_+ = \mathcal{S}^{d-1} \cap \RR_+^d$.
We now consider an $\ets$  distribution on $\RR^{d+1}$ with independent components.

If $\nu$ is a one-dimensional $\ets$ L\'evy measure from equation \eqref{nuTstable} we have 
\begin{equation*}
    \nu(dx) = \left(\boldsymbol{1}_{(0,\infty)}(x)\frac{e^{-\beta^+ |x|}}{|x|^{\alpha+1}}\lambda^++\boldsymbol{1}_{(-\infty,0)}(x)\frac{e^{-\beta^- |x|}}{|x|^{\alpha+1}}\lambda^-\right)dx,
\end{equation*}
where $\beta^+=\beta(1)$, $\beta^-=\beta(-1)$, $\lambda^+=\lambda(\{1\})$ and $\lambda^-=\lambda(\{-1\})$. Thus, $\ets$ distributions are  a multivariate version of the  tempered stable distributions studied in \cite{kuchler2013tempered}, with the restriction that $\alpha$ is constant. The condition $\lambda$ has support on $\mathcal{S}^{d-1}_+ = \mathcal{S}^{d-1} \cap \RR_+^d$ implies  
\begin{equation}\label{nuTstable2}
\nu(dx)=\boldsymbol{1}_{(0,\infty)}(x)\frac{e^{-\beta x}}{x^{\alpha+1}}\lambda dx.
\end{equation}
If $\mu$ has L\'evy measure \eqref{nuTstable2} we write $\mu\sim\ets(\alpha, \beta, \lambda)$.
Let $\SS=(S_1,\ldots, S_{d+1}),$ with independent components and $S_j \sim \ets(\alpha, \beta_j, \lambda_j)$, $j \in \{1,\dots,d+1\}$, $\alpha\in (0,1)$, and let $\SS(t)$ be the zero drift Sato process such that $\SS(1)\overset{\mathcal{L}}{=}\SS$.

The process $\SS(t)$ has a characteristic function of the form
\begin{equation*}
\begin{split}
\hat{\mu}^S_{t}( \xxi)=&\prod_{j=1}^{d+1}\exp\{\Gamma(-\alpha)\lambda_j[(\beta_j -it^\rho\xi_j)^{\alpha}-\beta_j^{\alpha}]\}.
\end{split}
\end{equation*}

By independence of the increments of $\SS(t)$, the characteristic function $\hat{\mu}^S_{u,t}(\xi)$ of the increment $\SS(t)-\SS(u)$ is given by 
\[
\hat{\mu}^S_{u,t}(\xi)=\frac{\hat{\mu}^S_{t}(\xi)}{\hat{\mu}^S_{u}(\xi)}.
\]
\begin{example}
We specify the  Sato subordinator $\SS(t)$ to have unit time marginal  inverse Gaussian (IG) distributions.

Let $S_j$ have inverse Gaussian distribution (IG) with parameters:
\begin{equation}\label{sub}
S_{j}\sim IG\left(\frac{1}{2},\beta_j,  \lambda_j\right)
,\,j=1,...,d+1.
\end{equation}

Let now $\SS(t)$ be the Sato subordinator  with unit time distribution of $\SS$, its time $t$ characteristic   function $\hat{\mu}^{\SS}_{t}( \xxi)$ has the form
\begin{equation*}
    \hat{\mu}^{\SS}_{t}( \xxi)
    =
    \prod_{j=1}^{d+1} \exp\{ -2\sqrt{\pi} \lambda_j [\sqrt{\beta_j-it^\rho\xi_j} -\sqrt{\beta_j}] \}.
\end{equation*}

\end{example}
\subsection{Sato subordination of multiparameter OU processes}
To construct the subordinated process and apply Theorem \ref{Li}, we need to regularize the factor-based Sato process. In fact, if $0<\rho<1$, $\nu_F(t,\cdot)$ explodes for $t\rightarrow 0$. We define $\tilde{S}_j(t) = S_j(t+t_0) - S_j(t_0)$  and we choose $t_0>0$ for $0<\rho\leq1$ and $t_0\geq0$ for $\rho>1$, see \cite{li2016additive} for more details. The characteristic function of the regularized Sato process is 
$\hat{\mu}^S_{t}(\xi)={\hat{\mu}^S_{t+t_0}(\xi)}/{\hat{\mu}^S_{t_0}(\xi)}.$
Notice that the parameter $t_0$ can be fixed in advance and does not need to be calibrated.

In what follows, we assume that $\SS(t)$ is a regularized Sato process. When explicit computations are provided, for simplicity we assume $\rho>1$ and choose $t_0=0$. 

\begin{definition} \label{def:MultiOU2}
    The $\RR^{d}$-valued process defined by    
    \begin{equation}\label{eq:MultiOU2}
        \YY(t)=
        \begin{pmatrix}
            U_1(S_1(t)) +  a_1 U_{d+1}(S_{d+1}(t))
            \\
            \vdots\\
            U_d(S_d(t)) + a_d U_{d+1}(S_{d+1}(t))
        \end{pmatrix},
    \end{equation}
    where $\SS(t)=(S_1(t),\ldots, S_{d+1}(t))$ is a (regularized) Sato subordinator with independent components, is a factor-based Sato Ornstein-Uhlenbeck (Sato-OU) process.
\end{definition}
The proposed construction builds on a simple factor structure widely used in financial applications to model the common and idiosyncratic components of returns. In the Lévy setting, i.e. when the multiparameter process $\UU$ is a Lévy process, the resulting process $\YY(t)$ is an additive process. If, instead, $\UU$ is of OU type, the resulting process is, in general, no longer Markovian. Nevertheless, we can rely on the fact that $\YY(t; d+1)$ keeps the Markov property.

Let us consider the process $\UU(\ss;d+1)=(U_1(s_1), \ldots, U_d(s_d), U_{d+1}(s_{d+1}))$, and let $\YY(t;d+1) = \UU(\SS(t);d+1)$.
By construction, the process $\YY(t;d+1)$  is a time-inhomogeneous Markov process, whose generator follows from Corollary \ref{cor:generatorind}.
  For $\xx \in \RR^{d+1}$,
    \begin{equation} \label{eq:Generator_MultiOU2}
        \mathcal{G}_t f(\xx) =
        \sum_{j=1}^{d+1}\int_{\RR_+}\left(\int_{\RR} f(x_1,\ldots, \underbrace{y}_{j\text{-th}},\ldots, x_{d+1}) u^{j}_{s_j}(x_j, y) d y - f(\xx)\right) \nu_j(t, ds_j)
    \end{equation}
    for $f \in \cap_{j=1}^{d+1} D(\GGG^{(j)})$, where $D(\GGG^{(j)})$ are the domains of the marginal generators, and where $u^j_{s_j}(x, y)$, for each $j \in \{1,\ldots, d+1\}$, is the transition density of the one-dimensional OU process $U_j(t)$ in \eqref{eq:SDE_Uj} and \eqref{eq:SDE_U},  
    and 
    \begin{equation}\label{eq:Levj}
        \nu_j(t, ds_j)
        =
        \lambda_j\rho(t+t_0)^{\alpha\rho-1} \frac{\beta_js_j(t+t_0)^{-\rho}+\alpha}{s_j^{\alpha+1}} \exp\left\{{-\beta_j s_j(t+t_0)^{-\rho}}\right\} ds_j.
    \end{equation}
Since
the process $\YY(t)$ defined in Definition \ref{def:MultiOU2} is a linear transformation of the Markov process $\YY(t;d+1)$, we can use the generator of $\YY(t;d+1)$ given in~\eqref{eq:Generator_MultiOU2} or its transition semigroup to find some statistics of the process.
Indeed, $\YY(t) = {A}{\YY(t;d+1)}$, where $A$ is a $d \times (d+1)$ real matrix such that $A=(I || \aa)$, where $I$ is the $d \times d$ identity matrix, $\aa = (a_1,\dots,a_d)$, and $||$ denotes the row concatenation.
The generator in~\eqref{eq:Generator_MultiOU2} is still relevant for the process $\YY(t)$.
{Indeed, the integral expression for the generator in~\eqref{eq:Generator_MultiOU2} extends to functions $f=\varphi\circ A$ with $\varphi\in C_c^\infty(\RR^d)$, for which the integrals are absolutely convergent.}
In this case,

\begin{equation}
{\footnotesize
\begin{split}
    &\GGG_t \varphi (A\xx)
    =
    \int_{\RR_+}\left(\int_{\RR} \varphi(x_1+a_1z,\dots, x_d+a_d z) u^{d+1}_{s_{d+1}}(x_{d+1}, z) d z - \varphi (A(\xx))\right) \nu_{d+1}(t, ds_{d+1})
    \\
    &+
    \sum_{j=1}^{d}\int_{\RR_+}\left(\int_{\RR} \varphi(x_1+a_1x_{d+1},\dots, y+a_jx_{d+1},\dots, x_d+a_d x_{d+1}) u^{j}_{s_j}(x_j, y) d y - \varphi (A\xx)\right) \nu_j(t, ds_j).
\end{split}}
\end{equation}
Relying on the Markov property of $\YY(t; d+1)$  we can also find the characteristic function 
of the increment of $\YY(t_2)-\YY(t_1)$ conditional on $\YY(t_1;d+1)$.    
By Proposition \ref{prop:CharMP-OU}, we find the characteristic function of the increment of
$\YY(t_2;d+1) - \YY(t_1;d+1)$, that  is given by, for $\xx, \xxi \in \RR^{d+1}$,
    \begin{equation}\label{eq:chinc}
    \hat{\mu}^{\YY(\cdot;d+1)}_{t_1,t_2}(\xx, \boldsymbol{\xi})
    =
    \int_{\RR^{d+1}_+} \prod_{j=1}^{d+1} \hat{\mu}^j_{s_j}(x_j, \xi_j) \pi_{t_1,t_2}(d\ss)
    =
    \prod_{j=1}^{d+1} \int_{\RR_+} \hat{\mu}^j_{s_j}(x_j, \xi_j) \pi^j_{t_1,t_2}(ds_j),
    \end{equation}
    where $\hat{\mu}^j_{s_j}(x_j, \xi_j)$ is the characteristic function of $U_j(s_j)-x_j$, for every $j \in \{1,\dots,d+1\}$, $\xx = (x_1,\dots,x_{d+1})$, $\xxi = (\xi_1, \dots, \xi_{d+1})$, and where $\pi^{j}_{t_1,t_2}(ds_j)$, for $j \in \{1,\dots,d+1\}$, is the law of the increment $S_j(t_2) - S_j(t_1)$.
  
Since $\YY(t) = A \YY(t;d+1)$, we have, for $\xxi \in \RR^{d}$ and $\xx \in \RR^{d+1}$,
\begin{equation}
\begin{split}
    \mathbb{E}&[e^{i \xxi \cdot (\YY(t_2) - \YY(t_1))} \vert \YY(t_1;d+1) = \xx] 
    = 
    \mathbb{E}[e^{i \xxi \cdot A(\YY(t_2;d+1) - \YY(t_1;d+1))} \vert \YY(t_1;d+1) = \xx]
    \\
    &=
    \mathbb{E}[e^{i (A^{\top}\xxi) \cdot (\YY(t_2;d+1) - \YY(t_1;d+1))} \vert \YY(t_1;d+1) = \xx]
    \\
    &=
    \hat{\mu}^{\YY(\cdot;d+1)}_{t_1,t_2}(\xx, A^{\top}\boldsymbol{\xi})
    \\
    &=
    \prod_{j=1}^d \int_{\RR^+} \hat{\mu}^{j}_{s_j}(x_j,\xi_j) \pi^{j}_{t_1,t_2}(ds_j) 
    \int_{\RR^+} \hat{\mu}^{d+1}_{s_{d+1}} \left( x_{d+1},\sum_{j=1}^da_j\xi_j \right) \pi^{d+1}_{t_1,t_2}(ds_{d+1}).
\end{split}
\end{equation}

The latter expression is useful for practical applications, such as calibration of the model, simulation of sample paths and pricing. 

\begin{proposition} 
    The processes $\YY(t;d+1)$ and $\YY(t)$ in \eqref{eq:MultiOU2}, have bounded variation if and only if $\alpha\in(0,\tfrac{1}{2})$. 
\end{proposition}
\begin{proof}
Consider the increment law $u_{\ss}(\xx, \xx + d\yy)$ of the $\RR^{d+1}$-valued OU process. 
We have 
\begin{equation}
\begin{split}
\int_{\RR^{d+1}}(1\wedge |\yy|) \nu_{\YY}(t,\xx,  d\yy)= \int_{\RR^{d+1}_+}\left(\int_{\RR^{d+1}}(1\wedge |\yy|)u_{\ss}(\xx, \xx+d\yy)\right)\nu(t, ds),
\end{split}
\end{equation}  
and, by Lemma 30.3 in \cite{sa}, $\int_{\RR^{d+1}}(1\wedge |\yy|)u_{\ss}(\xx, \xx+d\yy)<C|s|^{1/2}$. By  Proposition 2.1 in \cite{semeraro2022multivariate}, if $\alpha\in (0,1/2)$ then $\YY(t; d+1)$ (and therefore $\YY(t)$) has bounded variation. Conversely, assume $\alpha\geq1/2$ and let $U(t)$ be an OU process with $\sigma>0.$
We have
\begin{equation}
    E_{\xx}[\boldsymbol{1}\wedge\|U(s)-x\|]\sim \sigma\sqrt{2/\pi}\sqrt{s}\,\,\, s\rightarrow 0,
\end{equation}
uniformly on compacts.
Therefore, there exists $c$ such that 
\begin{equation}
    E_{\xx}[\boldsymbol{1}\wedge\|U(s)-x\|]> c\sqrt{s}\,\,\, s\rightarrow 0.
\end{equation}
We have
\begin{equation}
\begin{split}
    \int_{\RR}(1\wedge |y|) \nu_{\YY}(t,x,  dy)
    &=
    \int_{\RR_+} \left(\int_{\RR}(1\wedge |y|) u_{s} (x, x+dy) \right)
    \nu(t, ds) 
    \\
    &\geq c\int_0^{\epsilon_c} \sqrt{s} \nu(t, ds)
    =
    \infty,
\end{split}
\end{equation}  
if $\alpha \geq 1/2$, since $\nu(t, ds)\sim s^{-1-\alpha}$. 
By independence of the components of $\YY(s, d+1)$, the assertion follows.
Given the factor structure of the subordinator,  taking a linear transformation does not affect the regularity of the trajectories.
\end{proof}

\section{Future work} \label{sec:FutureWorks}

In this work, we have investigated general additive time-changes of multidimensional and multiparameter Markov processes using the composition of commuting one-parameter semigroups. 
In the applications, we focused on the specific case of Gaussian Ornstein–Uhlenbeck processes time-changed by a Sato process.
This process is a multivariate extension of the process in \cite{li2013ornstein} designed to reproduce the factor-based dependence structure introduced in \cite{LuciSem1} for L\'evy processes and {extended} in \cite{semeraro2022multivariate} for additive ones {to model time dependent correlations}. 
The  multivariate process here proposed  is therefore able to capture mean reversion and time inhomogeneity while allowing for a flexible dependence structure. Indeed, the factor based structure provides a good correlation fit in both  the historical (\cite{luciano2016dependence}, \cite{semeraro2022multivariate}) and in the risk neutral (\cite{amici2025multivariate}) settings.
However, since the factor-based dependence structure was proposed in the L\'evy or additive setting, the applications developed so far have focused on the stock market.
Our future research aims at developing a factor based Sato-OU process for multi-asset commodity and short rate models. In fact, the class of processes we have introduced incorporates several key features that are relevant in these markets, namely: mean reversion, time inhomogeneity and a factor structure, see e.g.  \cite{meyer2008multi}, \cite{benth2007non}, \cite{ascione2026wasserstein} and \cite{deng2005levy}. Finally, since the construction is based on a linear transformation of a Markov process, it offers good analytical tractability for calibration, simulation and pricing.

\bibliographystyle{apalike}
\bibliography{Biblio}

\end{document}